\documentclass{article}
\usepackage[T1]{fontenc}
\usepackage[utf8]{inputenc}         
\usepackage{amssymb,amsmath,amsthm}
\usepackage[hmargin=2.5cm,vmargin=3.8cm]{geometry}
\usepackage{tikz}
\usepackage{float}
\usepackage{graphicx} 
\usepackage{enumerate}
\usepackage{enumitem, hyperref}
\usepackage{authblk}
\usepackage{xspace}
\usepackage{relsize}
\usepackage[procnumbered,linesnumbered,ruled,vlined]{algorithm2e}
\usepackage[numbers,sort]{natbib}
\usepackage{todonotes}
\usepackage{thm-restate}
\usepackage{comment}
\usepackage{apxproof}
\usepackage{tabularray}

\tikzstyle{noeud}=[circle,inner sep=2, minimum size =3 pt, line width = 1pt, draw=black, fill=white]
\tikzstyle{bignoeud}=[circle,inner sep=2, minimum size =18 pt, line width = 1pt, draw=black, fill=white]

\newtheorem{theorem}{Theorem}

\newtheorem{lemma}[theorem]{Lemma}

\theoremstyle{definition}

\newtheorem{remark}[theorem]{Remark}

\newtheorem{claim}{Claim}
\newtheorem{problem}{Open Problem}

\usepackage{cleveref}

\newcommand{\tmax}{t_{\max}}
\DeclareMathOperator{\dist}{dist}
\DeclareMathOperator{\diam}{diam}
\DeclareMathOperator{\rea}{\mathbf{R}}

\newcommand{\temporalMD}{\textsc{Temporal Resolving Set}\xspace}
\newcommand{\TDM}{\textsc{3-Dimensional Matching}\xspace}

\newcommand{\decisionpb}[4]{
	\begin{center}
		\noindent\framebox{\begin{minipage}{#4\textwidth}
				#1\\
				Instance: #2\\ 
				Question: #3
		\end{minipage}}
	\end{center}
}

\newcommand{\longversion}[1]{}

\title{Resolving Sets in Temporal Graphs\thanks{This work was supported by the International Research Center "Innovation Transportation and Production Systems" of the I-SITE CAP 20-25 and by the ANR project GRALMECO (ANR-21-CE48-0004). The first author was also funded by the European Union (ERC, POCOCOP, 101071674). Views and opinions expressed are however those of the author(s) only and do not necessarily reflect those of the European Union or the European Research Council Executive Agency. Neither the European Union nor the granting authority can be held responsible for them. The research of the third author was supported by Business Finland Project 6GNTF, funding decision 10769/31/2022 and by the Academy of Finland, grant 338797.}~\thanks{A shorter conference version of this manuscript appeared in the proceedings of IWOCA 2024~\cite{bok2024resolving2}.}}

\author[1,2]{Jan Bok}
\author[1,3]{Antoine Dailly}
\author[1,4,5,6]{Tuomo Lehtilä}

\affil[1]{Université Clermont Auvergne, CNRS, Mines de Saint-Étienne, Clermont-Auvergne-INP, LIMOS, 63000, Clermont-Ferrand, France}
\affil[2]{Department of Algebra, Faculty of Mathematics and Physics, Charles University, Prague, Czechia}
\affil[3]{Université Clermont Auvergne, INRAE, UR TSCF, 63000, Clermont-Ferrand, France}
\affil[4]{Department of Mathematics and Statistics, University of Turku, Turku, Finland}
\affil[5]{Department of Computer Science, University of Helsinki, Helsinki, Finland}
\affil[6]{Helsinki Institute for Information Technology (HIIT), Espoo, Finland}

\date{}

\begin{document}
	
	\maketitle

	\begin{abstract}
		A \emph{resolving set} $R$ in a graph $G$ is a set of vertices such that every vertex of $G$ is uniquely identified by its distances to the vertices of $R$. Introduced in the 1970s, this concept has been since then extensively studied from both combinatorial and algorithmic points of view. We propose a generalization of the concept of resolving sets to temporal graphs, \emph{i.e.}, graphs with  edge sets that change over discrete time-steps. In this setting, the \emph{temporal distance from $u$ to $v$} is the earliest possible time-step at which a journey with strictly increasing time-steps on edges leaving $u$ reaches $v$, \emph{i.e.}, the first time-step at which $v$ could receive a message broadcast from $u$. A \emph{temporal resolving set} of a temporal graph $\mathcal{G}$ is a subset $R$ of its vertices such that every vertex of $\mathcal{G}$ is uniquely identified by its temporal distances from vertices of $R$.
		We study the problem of finding a minimum-size temporal resolving set, and show that it is NP-complete even on very restricted graph classes and with strong constraints on the time-steps: temporal complete graphs where every edge appears in either time-step~1 or~2, temporal trees where every edge appears in at most two consecutive time-steps, and even temporal subdivided stars where every edge appears in at most two (not necessarily consecutive) time-steps. On the other hand, we give polynomial-time algorithms for temporal paths and temporal stars where every edge appears in exactly one time-step, and give a combinatorial analysis and algorithms for several temporal graph classes where the edges appear in periodic time-steps.
		
		\medskip
		\noindent {\bf Temporal graphs, Resolving sets, Metric dimension, Trees, Graph algorithms, Complexity}  
	\end{abstract}
	
	\section{Introduction}
	
	For a set $R$ of vertices of a graph $G$, every vertex of $G$ can compute a vector of its distances from the vertices of $R$ (the distance, or number of edges in a shortest path, from $u$ to $v$ will be denoted by $\dist(u,v)$). If all such computed vectors are unique, then we call $R$ a \emph{resolving set} of $G$. This notion was introduced in the 1970s and gave birth to the notion of \emph{metric dimension of $G$}, that is, the smallest size of a resolving set of $G$. Metric dimension is a well-studied topic, with both combinatorial and algorithmic results, see for example surveys~\cite{kuziak2021metric,tillquist2023getting}.
	
	A \emph{temporal graph} can be defined as a graph on a given vertex set, and with an edge set that changes over discrete time-steps. Their study gained traction as a natural representation of dynamic, evolving networks~\cite{casteigts2012time,holme2015modern,holme2019temporal,michail2016introduction}. However, in the temporal setting, the notion of distance differs from the static setting: two vertices can be topologically close, but the \emph{journey} (\emph{i.e.}, a path in the underlying graph with strictly increasing time-steps\footnote{Those are sometimes called \emph{strict} journeys in the literature, but as argued in~\cite{kunz2023which}, strict journeys are naturally suited to applications where one cannot traverse multiple edges at the same time.}) between them can be long, or even impossible.
	The smallest time-step at which a journey from $u$ reaches $v$ is called the \emph{temporal distance from $u$ to $v$} and is denoted by $\dist_t(u,v)$. Note that in a temporal graph, there might be vertex pairs $(u,v)$ such that there is no temporal journey from $u$ to $v$, in which case we set the temporal distance as infinite. Furthermore, the temporal distance of non-adjacent vertices is not necessarily symmetric.

	The notion of temporal distance allows us to define a \emph{temporal resolving set} as a set $R$ of vertices of a temporal graph such that every vertex has a unique vector of temporal distances from the vertices of $R$. We are interested in the problem of finding a minimum-size temporal resolving set.
	
	Our motivation is both introducing a temporal variant of the well-studied problem of resolving sets and studying its combinatorial and algorithmic properties, as well as the problem of location in dynamic networks. Indeed, resolving sets are an analogue of geolocation in discrete structures~\cite{tillquist2023getting}, and thus temporal resolving sets are similar: if we consider that transmitters placed on vertices of the temporal resolving set emit continuously, we can locate ourselves by waiting long enough to receive the signals and constructing the temporal distance vector.
	
	In the remainder of this section, we  give an overview of the static (\emph{i.e.}, non-temporal) version and some variants of resolving set as well as an overview of temporal graphs, before giving a formal definition of the \temporalMD problem and an outline of our results.
	
	\medskip
	\noindent\textbf{Separating vertices.} Standard resolving sets were introduced independently by Harary and Melter~\cite{harary1976metric} and by Slater~\cite{slater1975leaves}, and have been well-studied due to their various applications (robot navigation~\cite{khuller1996landmarks}, detection in sensor networks~\cite{slater1975leaves}, and more~\cite{tillquist2023getting}). Their non-local nature makes them difficult to study from an algorithmic point of view: finding a minimum-size resolving set is NP-hard even on very restricted graph classes (planar graphs of bounded degree~\cite{diaz2017complexity}, bipartite graphs~\cite{epstein2015weighted}, and interval graphs of diameter~2~\cite{foucaud2017identification}, to name a few) and W[2]- and W[1]-hard when parameterized by solution size~\cite{hartung2013parameterized} and feedback vertex set~\cite{galby2023metric}, respectively. On the positive side, the problem is polynomial-time solvable on trees~\cite{chartrand2000resolvability,harary1976metric,khuller1996landmarks,slater1975leaves}, outerplanar graphs~\cite{diaz2017complexity} and cographs~\cite{epstein2015weighted}, to name a few. 
 
    Note that there are two conditions for a set to be resolving: \emph{reaching} (every vertex must be reached from some vertex of the resolving set) and \emph{separating} (no two vertices may have the same distance vector). A \emph{weak resolving set} is a set that is separating, but not necessarily reaching, all the vertices of a graph (alternatively, some authors, \emph{e.g.}~\cite{gutkovich2023computing}, call a reaching and separating set a \emph{dominating resolving set}). Minimum-size standard and weak resolving sets can thus differ in size by at most~1 as only one vertex can be unreachable.
	
	A natural variant of resolving sets consists in limiting the distance at which transmitters can emit. A first constraint is that of a robot which can only perceive its direct neighborhood: an \emph{adjacency resolving set}~\cite{jannesari2012metric} is a resolving set using the following distance: $\dist_a(u,v)=\min(\dist(u,v),2)$. More generally, a \emph{$k$-truncated resolving set}~\cite{estrada2021k} is a resolving set using the following distance: $\dist_k(u,v)=\min(\dist(u,v),k+1)$ (so an adjacency resolving set is a 1-truncated resolving set). Note that both variants were mostly studied in their weak version. The $k$-truncated resolving sets have quickly attracted attention on the combinatorial side~\cite{bartha2023sharp,frongillo2022truncated,geneson2021distance,geneson2022broadcast}. On the algorithmic side, the corresponding decision problem is known to be NP-hard~\cite{estrada2021k}, even on trees (although, in this case, it becomes polynomial-time solvable when $k$ is fixed) ~\cite{gutkovich2023computing}. Another notable variant requires the resolving set to also be a dominating set~\cite{brigham2003resolving}, and has been applied on temporal graph setup for machine learning applications~\cite{muklisin2023analysis}.

	\medskip
	\noindent\textbf{Temporal graphs.}
	A temporal graph $\mathcal{G}=(V,E_1,\ldots,E_{\tmax})$ is described by a sequence of edge sets representing the graph at discrete \emph{time-steps}, which are positive integers in $\{1,\ldots,\tmax\}$~\cite{casteigts2012time} (note that the sequence might be infinite, in which case the number of time-steps is not bounded by $\tmax$, and we can adapt all definitions in consequence).
	An alternate, equivalent description of $\mathcal{G}$ is $\mathcal{G}=(V,E,\lambda)$ (or $(G,\lambda)$ where $G=(V,E)$) with $E = \bigcup_{i=1}^{\tmax}E_i$ is called the \emph{underlying graph} and $\lambda: E \rightarrow 2^{\{1,\ldots,\tmax\}}$ is an edge-labeling function called \emph{time labeling} such that $\lambda(e)$ is the set of time-steps at which the edge $e$ exists, \emph{i.e.}, $i \in \lambda(e)$ if and only if $e \in E_i$~\cite{kempe2000connectivity}. We call a time labeling a \emph{$k$-labeling} if $|\lambda(e)| \leq k$ for every edge $e$. Furthermore, we say that a temporal graph is a \emph{temporal tree} (resp. \emph{temporal star}, etc.) if its underlying graph, understood as the graph induced by the union of all its edge sets, is a tree (resp. star, etc.). For a subgraph $G'$ of graph $G$ and a time labeling $\lambda$ of $G$, we denote by $\lambda_{|G'}$ the restriction of $\lambda$ to subgraph $G'$.
	
	A specific case of temporal graphs are those with a repeating sequence of edge sets, which have been studied in contexts such as routing~\cite{flocchini2012searching,flocchini2013exploration,ilcinkas2011power,liu2008scalable}, graph exploration~\cite{bellitto2023restless}, cops and robbers games~\cite{erlebach2020game,de2023cops} and others~\cite{arrighi2023multi,zschoche2020complexity} due to their natural applications in, \emph{e.g.}, transportation networks. Formally, a \emph{$p$-periodic $k$-labeling}
	is a time labeling such that both $E_{i+p}=E_i$ for every $i \geq 1$ and $|\lambda(e) \cap \{1,\ldots,p\}| \leq k$ for every edge $e$. A temporal graph with a periodic time labeling has an infinite sequence of edge sets, but can be represented with its $p$  first time-steps, understanding that the sequences will repeat after this.
	
	A vertex $v$ is said to be \textit{reachable from} another vertex $u$ if there exists a journey from $u$ to $v$. For a given vertex $v$, the set of vertices which can be reached from vertex $v$ is denoted by $\rea(v)$. For a vertex set $S$ such that $v\in S$, we denote by $\rea^S(v)\subseteq \rea(v)$ the set of vertices which can be reached from $v$ but not from any other vertex in $S$.
	
	\medskip
	\noindent\textbf{Temporal resolving sets.} We extend the definition of resolving sets to the temporal setting: a resolving set in a temporal graph is a reaching and separating set. More formally, a set $R$ of vertices of a temporal graph $\mathcal{G}=(V,E,\lambda)$ is a \emph{temporal resolving set} if $(i)$ for every vertex $v \in V$, there is a vertex $s \in R$ such that $v \in \rea(s)$; $(ii)$ for every two different vertices $u,v \in V$, there is a vertex $s \in R$ such that $\dist_t(s,u) \neq \dist_t(s,v)$. 
	Note that every vertex in a temporal resolving set is trivially separated from every other vertex.
	The problem we are studying is the following:
	
	\decisionpb{\temporalMD}{A temporal graph $\mathcal{G}=(V,E,\lambda)$; and an integer $k$.}{Is there a \emph{temporal resolving set} of size at most $k$?}{0.9}
	
	Due to the fact that temporal distance is not a metric in the usual sense (symmetry and the standard definition of the triangle inequality might not hold), we call the minimum size of a temporal resolving set of a graph $G$ the \emph{temporal resolving number of $G$} instead of temporal metric dimension. 
	
	Note that we can assume in the following that there is an edge $e$ such that $1 \in \lambda(e)$ (otherwise, let $m$ be the smallest time-step and decrease every time-step by $m-1$).
	Temporal resolving set can also be seen as a generalization of standard and $k$-truncated resolving sets: if $\lambda(e)=\{1,\ldots,\diam(G)\}$ for every edge $e$ (where $\diam(G)=\max\{\dist(u,v)~|~u,v \in V\}$, then a temporal resolving set is a standard resolving set; and if $\lambda(e)=\{1,\ldots,k\}$ for every edge $e$, then a temporal resolving set is a $k$-truncated resolving set\footnote{Also called a \emph{$k$-truncated dominating resolving set} in~\cite{gutkovich2023computing}.}.
	
	\medskip
	\noindent\textbf{Our results and outline.} We focus on time labelings with few labels per edge, mostly limiting ourselves to one or two labels per edge. Although the setting is more restricted than the general case, we prove that these scenarios already yield NP-complete problems or non-trivial polynomial algorithms.
	
	We present three sets of results. First, we focus in \Cref{sec-NPHARD} on the computational complexity of finding a minimum-size temporal resolving set in temporal graphs with 2-labelings. In particular, we prove that the problem is NP-complete on temporal complete graphs, which contrasts heavily with other resolving set problems. The problem is also NP-complete on temporal subdivided stars, and on temporal trees even when the two time-steps are consecutive.
	
	In \Cref{sec-POLY}, we give polynomial-time algorithms for some classes with 1-labelings. However, even for temporal paths, while the algorithm is quite natural, proving optimality is non-trivial. We also give algorithms for temporal stars, and for temporal subdivided stars when $\tmax=2$.
		
	Finally, in \Cref{SecPartLabel}, we take a more combinatorial approach to periodic time labelings. We find the optimal bounds for the temporal resolving number of several graph classes under this setting, namely in temporal paths, cycles, complete graphs, complete binary trees, and subdivided stars. We also prove that \temporalMD is FPT on trees with respect to number of leaves and XP on subdivided stars with $p$-periodic 1-labelings with respect to the period $p$.
	
	\section{NP-hardness of \temporalMD}
	\label{sec-NPHARD}
	
	In this section, we give hardness results for \temporalMD on very restricted graph classes, and with strong constraints on the time labeling.
	
	In the static setting, complete graphs tend to be easy to work with: since all the vertices are twins, they are indistinguishable from one another, and hence we need to take all of them but one in order to separate them. Indeed, the metric dimension and location-domination number of $K_n$ are both $n-1$.
	However, this is not the case with temporal complete graphs, since now the vertices are not necessarily twins anymore. We even prove that it is NP-hard:
	
	\begin{theorem}
		\label{thm-completeGraphsLabels12}
		\temporalMD is NP-complete on temporal complete graphs with a 1-labeling, even when there are only two time-steps.
	\end{theorem}
	
	\begin{proof}
		We reduce from the problem of finding a minimum-size adjacency resolving set, which is NP-hard on planar graphs~\cite{fernau2018adjacency}.
		
		Let $G=(V,E)$ be a connected planar graph, and $k$ be an integer. We construct the following temporal complete graph $\mathcal{G}=(V,E',\lambda)$:
		\begin{itemize}
			\item For every $e \in E$, $\lambda(e)=\{1\}$;
			\item For every pair of vertices $u,v$ such that $uv \not\in E$, $\lambda(uv)=\{2\}$.
		\end{itemize}
		An adjacency resolving set in $G$ clearly is a temporal resolving set in $\mathcal{G}$ and vice versa: the edges $e$ such that $\lambda(e)=\{2\}$ in $\mathcal{G}$ are exactly the paths of length at least~2 in $G$.
	\end{proof}
	
	The next two results are both proved by reducing from \TDM, one of the seminal NP-complete problems~\cite{karp2010reducibility}, and are inspired by the NP-completeness proof for $k$-truncated resolving sets on trees in~\cite{gutkovich2023computing}, with nontrivial adaptations to constrain the setting as much as possible.
	
	\decisionpb{\TDM (\textsc{3DM})}{A set $S\subseteq X\times Y\times Z$, where $X$, $Y$, and $Z$ are disjoint subsets of $\{1,\ldots,n\}$ of size $p$;
		and an integer $\ell < |S|$.}{Does $S$ contain a {\em matching} of size at least $\ell$, \emph{i.e.}, a subset $M\subseteq S$ such that $|M| \geq \ell$ and no two elements of $M$ agree in any coordinate?}{0.9}
	
	\begin{theorem}
		\label{thm-NPhardnessOnSubStars}
		\temporalMD is NP-complete on temporal stars, in which every edge is subdivided twice, with a $2$-labeling.
	\end{theorem}
	
	\begin{proof}
		We note that the problem is clearly in NP: a certificate is a set of vertices, and for each vertex, we can compute the time vectors and check that they are all different and that every vertex is reached by at least one vertex from the set in polynomial time. To prove completeness, we reduce from \textsc{3DM}.
		
		Starting from an instance $(S,\ell)$ of \textsc{3DM}, denoting $s=|S|$ and the $i$-th triple in $S$ by $(x_i,y_i,z_i)$ with $x_i \in X, y_i \in Y, z_i \in Z$, we construct an instance $(\mathcal{G},s+2-\ell)$ of \temporalMD, where $\mathcal{G}=(T,\lambda)$. This construction is detailed below, see \Cref{fig-NPHardnessSubStars}. 
		
		Let $V(T) = \{u,t_1,t_2,t_3\} \cup \mathlarger{\bigcup}_{i=1}^s \{a_i,b_i,c_i\}$. We will arrange the vertices in the following way as a twice subdivided star. The center vertex is $u$ and it is attached to vertices $a_i$ which are adjacent to vertices $b_i$ which are adjacent to vertices $c_i$ for each $1\leq i\leq s$. Moreover, $u$ is also adjacent to $t_1$ which is adjacent to $t_2$ and which is adjacent to $t_3$. Vertices $\{a_i,b_i,c_i\}$ correspond to elements $\{x_i,y_i,z_i\}$ so that $a_i<b_i<c_i$.
		Edges are labeled as follows: 
		\begin{itemize}
			\item For every $i \in \{1,\ldots,s\}$, $\lambda(ua_i) = \{2,a_i+4\}$;
			\item For every $i \in \{1,\ldots,s\}$, $\lambda(a_ib_i) = \{3,b_i+4\}$;   \item For every $i \in \{1,\ldots,s\}$, $\lambda(b_ic_i) = \{4,c_i+4\}$;
			\item We have $\lambda(ut_1) = \{2\}$, $\lambda(t_1t_2) = \{1\}$ and $\lambda(t_2t_3) = \{3\}$;
		\end{itemize}
		
		Observe that the constructed graph is a star whose every edge is subdivided exactly twice. Moreover, every edge has at most two labels.

		We shall now prove that we decide YES for \textsc{3DM} on $(S,\ell)$ if and only if we decide YES for \temporalMD on $((T,\lambda),s+1-\ell)$.
		
		$(\Rightarrow)$ Assume that $S$ contains a matching $M$ of size at least $\ell$. We construct the following set:
		$R = \bigcup_{i \not\in M} \{ a_i \} \cup \{t_1\}.$
		\longversion{$$R = \mathlarger{\bigcup}_{i \not\in M} \{ a_i \} \cup \{t_1\}.$$}
		Note that $R$ contains $t_1$ and each $a_i$ such that the corresponding element of $S$ is not in $M$. Furthermore, every vertex of $T$ is reached from $t_1$, so we need only consider the separation part.
		
		First observe that $t_1$ separates $u$ and each $t_i$. Furthermore, it reaches each vertex of type $a_i$ at moment $a_i+4$, vertices $b_i$ at moment $b_i+4$ and vertices $c_i$ at moment $c_i+4$. Consider then some $a_h\in R$. It reaches vertex $b_h$ at moment $3$ and $c_h$ at moment $4$. Hence, together with vertex $t_1$, vertex $a_h$ separates vertices $a_h,b_h$ and $c_h$ from all other vertices. Furthermore, vertex $a_h$ reaches other vertices of types $a_i,b_i$ and $c_i$ at the same moment as vertex $t_1$. Hence, we have uniquely separated vertices $u,t_1,t_2,t_3$ and every $a_h,b_h,c_h$ such that $a_h\in R$. Recall that sets $X,Y$ and $Z$ are disjoint. Thus, each vertex $a_i$ is separated from vertices of type $b_j$ for any $i$ and $j$ and the same is true for $a_i$ and $c_j$ as well as $b_i$ and $c_j$. Let us next consider when we might not separate $a_i$ from $a_j$ (the same argument holds for pairs $b_i,b_j$ and $c_i,c_j$). We may assume that $\{a_i,a_j\}\cap R=\emptyset$. Thus, corresponding elements belong to $M$. Therefore $a_i\neq a_j$ and hence, they are reached at different time moments from vertex $t_1$, a contradiction. Therefore, $R$ is a temporal resolving set of the claimed cardinality.
		
		$(\Leftarrow)$ Assume that there is a temporal resolving set $R$ of size at most $s-\ell+1$. Since every vertex must be reached from a vertex of $R$, we must have one of vertices $t_1,t_2$ or $t_3$ in $R$. Now, as in the previous case, only the $a$'s $b$'s and $c$'s must be reached and resolved. If each $a_i$, $b_i$ and $c_i$ is unique, then $t_1$ is a resolving set of size $1\leq s-\ell+1$. Moreover, if for example $a_i=a_j$ (similar argument holds for $b_i=b_j$ and $c_i=c_j$), then at most one of corresponding tuples can belong to the matching. Moreover, to separate $a_i$ and $a_j$, we need a vertex in resolving set to belong to one of the branches. Hence, we may choose as our matching $M$ the sets corresponding to branches which contain no members of the temporal resolving set. There are at least $\ell$ such branches and the claim follows.
	\end{proof}
	
	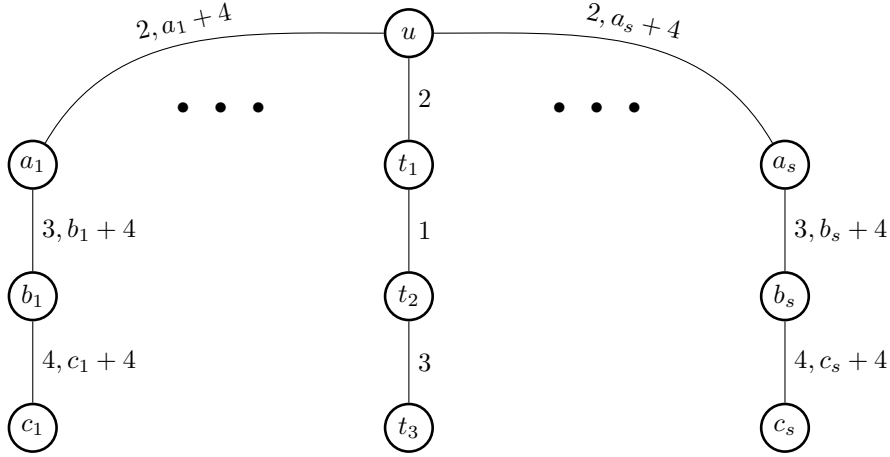
\begin{figure}[tb]
		\centering
			\begin{tikzpicture}
				
				\node[bignoeud] (u) at (7,12) {$u$};
				
				\node[bignoeud] (v11) at (2,10.25) {$a_1$};
				\node[bignoeud] (v12) at (2,8.5) {$b_1$};
				\node[bignoeud] (v13) at (2,6.75) {$c_1$};
				\node[bignoeud] (vs1) at (12,10.25) {$a_s$};
				\node[bignoeud] (vs2) at (12,8.5) {$b_s$};
				\node[bignoeud] (vs3) at (12,6.75) {$c_s$};
				
				\node[bignoeud] (t1) at (7,10.25) {$t_1$};
				\node[bignoeud] (t2) at (7,8.5) {$t_2$};
				\node[bignoeud] (t3) at (7,6.75) {$t_3$};
				
				\draw[out=180,in=60] (u) to node[midway,above,sloped]{$2,a_1+4$} (v11);
				\draw[out=0,in=120] (u) to node[midway,above,sloped]{$2,a_s+4$} (vs1);
				
				\draw (u) to node[midway,right]{$2$} (t1);
				\draw (t1) to node[midway,right]{$1$} (t2);
				\draw (t2) to node[midway,right]{$3$} (t3);
				
				\draw (v11) to node[midway,right]{$3,b_1+4$} (v12);
				\draw (v12) to node[midway,right]{$4,c_1+4$} (v13);
				\draw (vs1) to node[midway,right]{$3,b_s+4$} (vs2);
				\draw (vs2) to node[midway,right]{$4,c_s+4$} (vs3);
				
				\draw (4,11) node {$\bullet$};
				\draw (4.5,11) node {$\bullet$};
				\draw (5,11) node {$\bullet$};
				\draw (9,11) node {$\bullet$};
				\draw (9.5,11) node {$\bullet$};
				\draw (10,11) node {$\bullet$};
			\end{tikzpicture}
		\caption{The construction of the proof of \Cref{thm-NPhardnessOnSubStars}. Only the branches 1 and $s$ are detailed together with the control branch. We have $\{a_i,b_i,c_i\}=\{x_i,y_i,z_i\}$ where $a_i<b_i<c_i$.}
		\label{fig-NPHardnessSubStars}
	\end{figure}
	
	\begin{theorem}
		\label{thm-NPhardnessOnTrees}
		\temporalMD is NP-complete on temporal trees, even with only one vertex of degree at least~5, with a 2-labeling where the labels are consecutive.
	\end{theorem}

	\begin{proof}
		First, note that the problem is clearly in NP: a certificate is a set of vertices, and  for each vertex, we can compute the time vectors, check that they are all different and check that every vertex is reached by at least one vertex from the set in polynomial time. To prove completeness, we reduce from \TDM.
		
		Starting from an instance $(S,\ell)$ of \TDM, denoting $s=|S|$ and the $i$-th triple in $S$ by $(x_i,y_i,z_i)$ with $x_i \in X, y_i \in Y, z_i \in Z$, we will construct an instance $(\mathcal{G},\ell')$ of \temporalMD. This construction is detailed below. Note that the time-steps cover the interval $[n-1,n^2+1]$ in order to simplify the notations but, as discussed in the introduction, they can be brought down to the interval $[1,n^2-n+2]$ instead. Let:
		
		\begin{align*}
			V = &\{u\} \cup \mathlarger{\bigcup}_{i=1}^s \left\{ \{a_i,b_i,c_i\} \cup \left( \mathlarger{\bigcup}_{j=1}^n\{v_i^j,s_i^j,t_i^j\} \right)
			\cup 
   \left( \mathlarger{\bigcup}_{j=1}^{n-1}\mathlarger{\bigcup}_{k=1}^{n-1} \{w_i^{j,j+1,k},s_i^{j,j+1,k},t_i^{j,j+1,k}\} \right) \right\}.
		\end{align*}
		We will arrange the vertices in the following way, which will be formalized below: $u$ will be connected to every $v_i^1$ which will be the start of the $i$-th branch corresponding to the $i$-th element of $S$, every $t$ will be connected to its corresponding $s$ and either $v$ or $w$, all the $w_i^{j,j+1,k}$ will form a path linking $v_i^j$ and $v_i^{j+1}$, and the vertices $a_i,b_i,c_i$ will represent the tuple elements $x_i,y_i,z_i$. The edges and their labels are as follows (edges not described do not exist):
		\begin{itemize}
			\item For every $i \in \{1,\ldots,s\}$, $\lambda(uv_i^1) = \{n-1,n\}$;
			\item For every $i \in \{1,\ldots,s\}$, $\lambda(v_i^{x_i}a_i) = \{x_in,x_in+1\}$, $\lambda(v_i^{y_i}b_i) = \{y_in,y_in+1\}$, $\lambda(v_i^{z_i}c_i) = \{z_in,z_in+1\}$;
			\item For every $i \in \{1,\ldots,s\}$ and $j \in \{1,\ldots,n-1\}$, $\lambda(v_i^jw_i^{j,j+1,1})=\{jn,jn+1\}$ and $\lambda(w_i^{j,j+1,n}v_{j+1})=\{(j+1)n-1,(j+1)n\}$;
			\item For every $i \in \{1,\ldots,s\}$, $j \in \{1,\ldots,n-1\}$ and $k \in \{1,\ldots,n-2\}$, set $$\lambda(w_i^{j,j+1,k}w_i^{j,j+1,k+1})=\{jn+k,jn+k+1\};$$
			\item For every $i \in \{1,\ldots,s\}$ and $j \in \{1,\ldots,n\}$, set $$\lambda(s_i^jt_i^j)=\lambda(t_i^jv_i^j)=\{n^2+1\};$$
			\item For every $i \in \{1,\ldots,s\}$, $j \in \{1,\ldots,n-1\}$ and $k \in \{1,\ldots,n-1\}$, set $$\lambda(s_i^{j,j+1,k}t_i^{j,j+1,k})=\lambda(t_i^{j,j+1,k}w_i^{j,j+1,k})=\{n^2+1\}.$$
		\end{itemize}
		Note that the underlying graph $T$ thus constructed is a tree, and that every edge has, as time labels, an interval of size at most~2. Furthermore, $u$ is the only vertex with degree at least~5.
		This construction is depicted on \Cref{fig-NPHardnessTrees}.
		
		Let $\ell'=s(n(n-1)+1)+(s-\ell)$. We prove that we decide YES for \TDM on $(S,\ell)$ if and only if we decide YES for \temporalMD on $((T,\lambda),\ell')$.
		
		$(\Rightarrow)$ Assume that $S$ contains a matching $M$ of size at least $\ell$. We construct the following set: $$R = \mathlarger{\bigcup}_{i \not\in M} \{ v_i^1 \} \cup \mathlarger{\bigcup}_{i=1}^s \left\{ \mathlarger{\bigcup}_{j=1}^n \{ t_i^j \} \cup \mathlarger{\bigcup}_{j=1}^{n-1}\mathlarger{\bigcup}_{k=1}^{n} \{ t_i^{j,j+1,k} \} \right\},$$ \emph{i.e.}, $R$ contains every $t$ and all the first vertices of every branch of $T$ such that the corresponding element of $S$ is not in $M$. Note that every vertex of $T$ is reached from an element of $R$, so we need to consider the separation part. First, note that $u$ and each $s$, $t$, $v$, and $w$ is uniquely separated by $R$ ($u$ by the $v_i^1$'s we selected, the other ones by the $t$'s). Hence, the only possible vertices not separated by $R$ are $a$'s, $b$'s and $c$'s. Note that, by construction, $a_i$ (resp. $b_i$, $c_i$) will be reached at time $x_in+1$ (resp. $y_in+1$, $z_in+1$) from any vertex $v_j^1$ such that $j \neq i$, and at time $x_in$ (resp. $y_in$, $z_in$) from $v_i^1$. Hence, all the $a_i$'s, $b_i$'s and $c_i$'s in branches $i$ such that $i \not\in M$ are separated by $R$. Furthermore, for every branch $i$, the vertices $a_i$, $b_i$ and $c_i$ are separated from each other. Assume now that two vertices are not separated by $R$, they have to be $a_i$ and $a_j$ (without loss of generality) such that $i \neq j$ and $i,j \in M$. However, this is only possible if $x_i=x_j$, in which case the elements $i$ and $j$ from $M$ cannot be in the same matching, which is a contradiction. Hence, $R$ is a temporal resolving set of size $s(n(n-1)+1)+(s-\ell)$: it contains $(s-\ell)$ vertices $v_i^1$, and there are $n(n-1)+1$ vertices $t$ for each of the $s$ branches.
		
		$(\Leftarrow)$ Assume that there is a temporal resolving set $R$ of size at most $\ell'=s(n(n-1)+1)+(s-\ell)$. Since every vertex must be reached from a vertex of $R$, for every pair of adjacent $s$ and $t$, at least one of them must be in $R$. Without loss of generality, assume that every $t$ is in $R$ (since this allows to reach and separate every $v$ and $w$): this means that the number of non-$t$ vertices in $R$ is at most $s-\ell$ (since every branch contains $n(n-1)+1$ pairs of adjacent $s$ and $t$). Now, as in the previous case, only the $a$'s $b$'s and $c$'s must be reached and separated, as well as $u$. Selecting either $u$ or any $v_i^1$ will take care of $u$ and allow to reach the $a$'s, $b$'s and $c$'s. Again, the possible conflicts among those vertices are the ones such that (without loss of generality) $x_i=x_j$ for $i \neq j$. In this case, the only way to separate the pair would have been to either select $a_i$ or $a_j$, or to select any vertex above them in either (or both) of the branches $i$ and $j$. Since $R$ is a temporal resolving set, all such pairs have been separated. We construct $M$ the following way: add to $M$ every $i$ such that the only vertices of the branch $i$ in $R$ are its $s$'s and $t$'s. No two elements of $M$ can verify $x_i=x_j$ (resp. $y_i=y_j$, $z_i=z_j$), since that would imply that the corresponding pair $(a_i,a_j)$ (resp. $(b_i,b_j)$, $(c_i,c_j)$) would not be separated: no vertex of branch $i$ would have been in $R$, and thus $R$ would not be a temporal resolving set. Hence, $M$ is a matching of size at least $\ell$: at most $(s-\ell)$ branches contain a vertex of $R$ that is not a $t$, and thus at least $\ell$ branches do not.
	\end{proof}
	
	\begin{figure}[h]
		\centering
		\begin{tikzpicture}[scale=0.95]
			\node[bignoeud] (u) at (6,12) {$u$};
			
			\node[bignoeud] (a1) at (4.5,10) {$a_1$};
			\node[bignoeud] (b1) at (4.5,0) {$b_1$};
			\node[bignoeud] (c1) at (4.5,4) {$c_1$};
			
			\node[bignoeud] (v11) at (2,10) {$v_1^1$};
			\node[bignoeud] (v12) at (2,7) {$v_1^2$};
			\node[bignoeud] (v13) at (2,4) {$v_1^3$};
			\node[bignoeud] (v1n) at (2,0) {$v_1^n$};
			\node[bignoeud] (vs1) at (10,10) {$v_s^1$};
			
			\node[noeud] (w121) at (2,9) {};
			\draw (w121) node[right,scale=0.75,xshift=0.5mm] {$w_1^{1,2,1}$};
			\node[noeud] (w12n) at (2,8) {};
			\draw (w12n) node[right,scale=0.75,xshift=0.5mm] {$w_1^{1,2,n-1}$};
			\node[noeud] (w231) at (2,6) {};
			\draw (w231) node[right,scale=0.75,xshift=0.5mm] {$w_1^{2,3,1}$};
			\node[noeud] (w23n) at (2,5) {};
			\draw (w23n) node[right,scale=0.75,xshift=0.5mm] {$w_1^{2,3,n-1}$};
			\node[noeud] (w341) at (2,3) {};
			\draw (w341) node[right,scale=0.75,xshift=0.5mm] {$w_1^{3,4,1}$};
			
			\node[noeud] (t11) at (0.5,10) {};
			\node[noeud] (s11) at (-1,10) {};
			\draw (s11) node[below,scale=0.75,yshift=-0.5mm] {$s_1^1$};
			\draw (t11) node[below,scale=0.75,yshift=-0.5mm] {$t_1^1$};
			\node[noeud] (t12) at (0.5,7) {};
			\node[noeud] (s12) at (-1,7) {};
			\draw (s12) node[below,scale=0.75,yshift=-0.5mm] {$s_1^2$};
			\draw (t12) node[below,scale=0.75,yshift=-0.5mm] {$t_1^2$};
			\node[noeud] (t13) at (0.5,4) {};
			\node[noeud] (s13) at (-1,4) {};
			\draw (s13) node[below,scale=0.75,yshift=-0.5mm] {$s_1^3$};
			\draw (t13) node[below,scale=0.75,yshift=-0.5mm] {$t_1^3$};
			\node[noeud] (t1n) at (0.5,0) {};
			\node[noeud] (s1n) at (-1,0) {};
			\draw (s1n) node[below,scale=0.75,yshift=-0.5mm] {$s_1^n$};
			\draw (t1n) node[below,scale=0.75,yshift=-0.5mm] {$t_1^n$};
			\node[noeud] (ts1) at (8.5,10) {};
			\node[noeud] (ss1) at (7,10) {};
			\draw (ss1) node[below,scale=0.75,yshift=-0.5mm] {$s_s^1$};
			\draw (ts1) node[below,scale=0.75,yshift=-0.5mm] {$t_s^1$};
			
			\node[noeud] (t121) at (0.5,9) {};
			\node[noeud] (s121) at (-1,9) {};
			\draw (s121) node[below,scale=0.75,yshift=-0.5mm] {$s_1^{1,2,1}$};
			\draw (t121) node[below,scale=0.75,yshift=-0.5mm] {$t_1^{1,2,1}$};
			\node[noeud] (t12n) at (0.5,8) {};
			\node[noeud] (s12n) at (-1,8) {};
			\draw (s12n) node[below,scale=0.75,yshift=-0.5mm] {$s_1^{1,2,n-1}$};
			\draw (t12n) node[below,scale=0.75,yshift=-0.5mm] {$t_1^{1,2,n-1}$};
			\node[noeud] (t231) at (0.5,6) {};
			\node[noeud] (s231) at (-1,6) {};
			\draw (s231) node[below,scale=0.75,yshift=-0.5mm] {$s_1^{2,3,1}$};
			\draw (t231) node[below,scale=0.75,yshift=-0.5mm] {$t_1^{2,3,1}$};
			\node[noeud] (t23n) at (0.5,5) {};
			\node[noeud] (s23n) at (-1,5) {};
			\draw (s23n) node[below,scale=0.75,yshift=-0.5mm] {$s_1^{2,3,n-1}$};
			\draw (t23n) node[below,scale=0.75,yshift=-0.5mm] {$t_1^{2,3,n-1}$};
			\node[noeud] (t341) at (0.5,3) {};
			\node[noeud] (s341) at (-1,3) {};
			\draw (s341) node[below,scale=0.75,yshift=-0.5mm] {$s_1^{3,4,1}$};
			\draw (t341) node[below,scale=0.75,yshift=-0.5mm] {$t_1^{3,4,1}$};
			
			\draw[out=180,in=90] (u) to node[midway,above,sloped]{$n-1,n$} (v11);
			\draw[out=0,in=90] (u) to node[midway,above,sloped]{$n-1,n$} (vs1);
			
			\draw (v11) to node[midway,above]{$n,n+1$} (a1);
			\draw (v1n) to node[midway,above]{$n^2,n^2+1$} (b1);
			\draw (v13) to node[midway,above]{$3n,3n+1$} (c1);
			
			\draw (v11) to node[midway,right]{$n,n+1$} (w121);
			\draw (w12n) to node[midway,right]{$2n-1,2n$} (v12);
			\draw (v12) to node[midway,right]{$2n,2n+1$} (w231);
			\draw (w23n) to node[midway,right]{$3n-1,3n$} (v13);
			\draw (v13) to node[midway,right]{$3n,3n+1$} (w341);
			
			\draw[dashed] (w121) to (w12n);
			\draw[dashed] (w231) to (w23n);
			\draw[dashed] (w341) to (2,2);
			\draw[dashed] (v1n) to (2,1);
			\draw[dashed] (vs1) to (10,8.5);
			
			\draw (s11) to node[midway,above]{$n^2+1$} (t11);
			\draw (s12) to node[midway,above]{$n^2+1$} (t12);
			\draw (s13) to node[midway,above]{$n^2+1$} (t13);
			\draw (s1n) to node[midway,above]{$n^2+1$} (t1n);
			\draw (s121) to node[midway,above]{$n^2+1$} (t121);
			\draw (s12n) to node[midway,above]{$n^2+1$} (t12n);
			\draw (s231) to node[midway,above]{$n^2+1$} (t231);
			\draw (s23n) to node[midway,above]{$n^2+1$} (t23n);
			\draw (s341) to node[midway,above]{$n^2+1$} (t341);
			\draw (ss1) to node[midway,above]{$n^2+1$} (ts1);
			
			\draw (v11) to node[midway,above]{$n^2+1$} (t11);
			\draw (v12) to node[midway,above]{$n^2+1$} (t12);
			\draw (v13) to node[midway,above]{$n^2+1$} (t13);
			\draw (v1n) to node[midway,above]{$n^2+1$} (t1n);
			\draw (w121) to node[midway,above]{$n^2+1$} (t121);
			\draw (w12n) to node[midway,above]{$n^2+1$} (t12n);
			\draw (w231) to node[midway,above]{$n^2+1$} (t231);
			\draw (w23n) to node[midway,above]{$n^2+1$} (t23n);
			\draw (w341) to node[midway,above]{$n^2+1$} (t341);
			\draw (vs1) to node[midway,above]{$n^2+1$} (ts1);
			
			\draw (5.5,11) node {$\bullet$};
			\draw (6,11) node {$\bullet$};
			\draw (6.5,11) node {$\bullet$};
		\end{tikzpicture}
		\caption{The construction of the proof of \Cref{thm-NPhardnessOnTrees}. Only the branch 1 is detailed, we have $x_1=1$, $y_1=n$ and $z_1=3$. Dashed lines represent longer paths.}
		\label{fig-NPHardnessTrees}
	\end{figure}
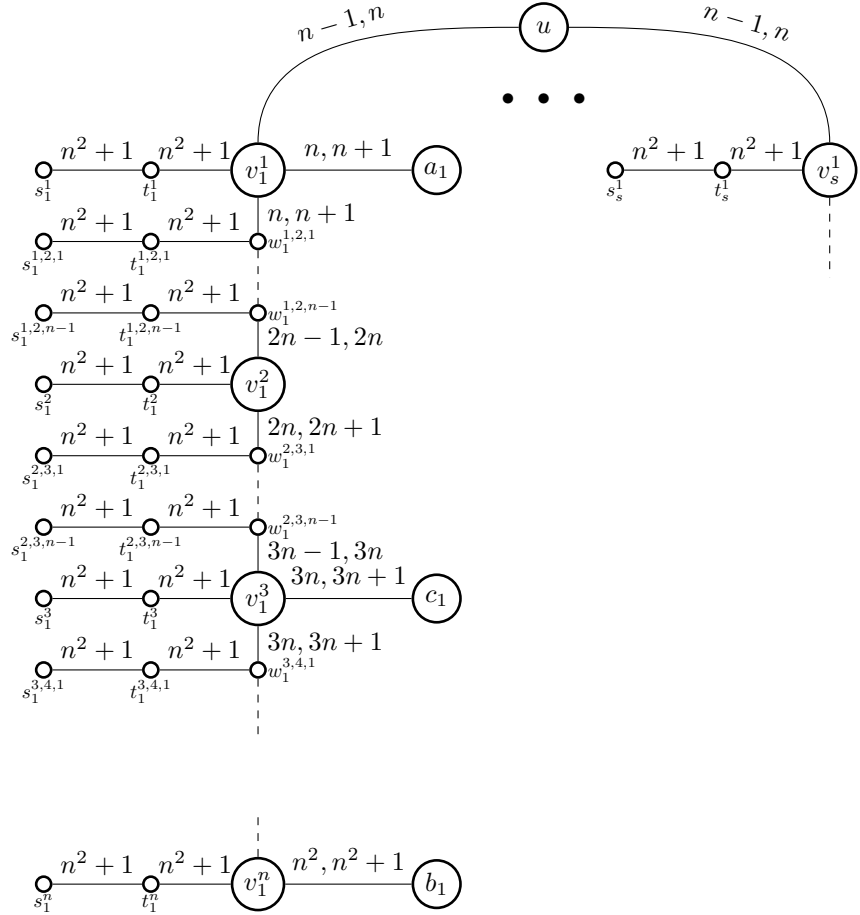
	
	\begin{remark}
		We can set $\lambda(t_i^jv_i^j)=\lambda(t_i^{j,j+1,k}w_i^{j,j+1,k})=\{n^2+2,n^2+3\}$ and $\lambda(s_i^jt_i^j)=\lambda(s_i^{j,j+1,k}t_i^{j,j+1,k})=\{n^2+1,n^2+2\}$, to obtain a construction with time intervals of size exactly~2.
	\end{remark}

	\section{Polynomial-time algorithms for subclasses of trees}
	\label{sec-POLY}
	
	In this section, we give polynomial-time algorithms for \temporalMD. We study temporal paths and stars with one time label per edge, and temporal subdivided stars with one time label per edge and where every label is in $\{1,2\}$. Recall that \temporalMD is already NP-complete on temporal subdivided stars with $2$-labeling (\Cref{thm-NPhardnessOnSubStars}), so these results are a first step for bridging the gap between polynomial-time and NP-hard.
	
	\subsection{Temporal paths}\label{sec-polyPaths}
	
	Throughout this subsection, we denote by $P_n$ a path on $n$ vertices $v_1,\dots,v_n$, with edges $v_iv_{i+1}$ for $1 \leq i \leq n-1$. Furthermore, we assume that $\lambda$ is a 1-labeling. \Cref{ALGTemporalPaths1label} constructs a minimum-size temporal resolving set $R$ of $\mathcal{P}=(P_n,\lambda)$.  The core of the algorithm consists in adding to $R$ the last vertex that can reach a leaf, then check if it separates everything in the two directions. If so, we can iterate on the vertices it cannot reach, and otherwise we have to add a vertex that separates the conflicting vertices before iterating. We denote $\lambda(v_iv_{i+1})$ by $t_i$ and the elements of $R$ by $\{r_1,\dots, r_{|R|}\}$, and we assume that if $r_i=v_j$ and $r_h=v_k$ for $h>i$, then $k>j$.  Consider vertex $v_i$, we say that $v_j$ is on its right (resp. left) side if $j>i$ (resp. $j<i$). The set of vertices on the left side of vertex $v_j$ is denoted by $\ell(v_j)$. 
	
\begin{algorithm}[th!]
	\caption{Temporal resolving set for temporal paths with 1-labeling}
	\label{ALGTemporalPaths1label}
	\SetKwInOut{KwIn}{Input}
	\SetKwInOut{KwOut}{Output}
	\KwIn{A temporal path $\mathcal{P}=(P_n,\lambda)$.}
	\KwOut{A minimum-size temporal resolving set $R$ of $\mathcal{P}$. }
 
	Set $v=v_1$ and $R=\emptyset$.\label{Alg1Setvv1}
 
 \While{true\label{Alg1While1}}{
 Set $s=v_i$ where $i$ is the largest integer such that $v_i$ reaches $v$. Add $s$ to $R$ and set $a=i$.\label{Alg1Setsvi}

  \eIf{within $\rea^R(s)$ every vertex has unique distance to $s$\label{Alg1If1}}
 {
  Set $w=v_j$ where $j$ is the smallest integer with $j>i$ for each $v_i\in \rea(s)$.\label{Alg1Then1}
  
 \If{$v_n\in \rea^R(s)$ or $v_n\in R$\label{Alg1If2}}{
 \Return $R$.\label{Alg1Then2}
 }
 }
 {\label{Alg1Else} Set $w=v_{b}$ where $v_b\in \rea^R(s)$ is the vertex which does not have unique distance among vertices in $\rea^R(s)$ to $s$ and among those vertices $b$ is minimal such that $b>a$.\label{Alg1Else2}}
 
   Let $v=w$.\label{Alg1Setvw}
}
   
	\end{algorithm}

	\begin{lemma}
		\label{lemSubpath}
		Let $\mathcal{P}=(P_n,\lambda)$ be a temporal path, $P_m$ be a subpath of $P_n$ containing one leaf of $P_n$, and $\mathcal{P}'=(P_m,\lambda_{|P_m})$. The temporal resolving number of $\mathcal{P}$ is at least as large as the temporal resolving number of $\mathcal{P}'$.
	\end{lemma}
	
	\begin{proof}
		We may assume without loss of generality that the temporal subpath $\mathcal{P}'$ contains vertices $v_a,v_{a+1},\dots, v_n$ for some $1\leq a\leq n$.
		Let $R\subseteq V(P_n)$ (resp. $R' \subseteq V(P_m)$) be a minimum-size temporal resolving set of $\mathcal{P}$ (resp. $\mathcal{P}'$). If $|R|\geq |R'|$, then the claim follows. Thus, assume by contradiction that $|R|< |R'|$. First observe that if $R\subseteq V(P_m)$, then $R$ is a temporal resolving set in $\mathcal{P}'$ and thus $|R|\geq|R'|$, a contradiction. Hence, we may assume that there exists some vertex $s\in R\setminus V(P_m)$. Let us consider the set $R''=\{v_a\}\cup R\cap V(P_m)$. Note that $|R''|\leq |R|$. First of all, every vertex in $\mathcal{P}'$ is reached by some vertex of $R''$. Secondly, if two vertices of $\mathcal{P}'$ are not separated by vertices in $R''\setminus\{v_a\}$, then they were separated by $s$ in $\mathcal{P}$. Moreover, in $R''$ they are separated by $v_a$. Indeed, if $w\in \rea(s)\cap V(P_m)$, then $w\in \rea(v_a)\cap V(P_m)$. Moreover, since $v_a$ is a leaf in $P_m$, every vertex in $\rea(v_a)\cap V(P_m)$ has a unique temporal distance to $v_a$. Therefore, $R''$ is a resolving set in $\mathcal{P}'$ with cardinality $|R''|<|R'|$, a contradiction. Thus, the claim follows.      
	\end{proof}
	
	\begin{theorem}
		\label{The:TempPath1label}
		There is an $\mathcal{O}(n)$ algorithm solving \temporalMD on temporal paths on $n$ vertices with a 1-labeling.
	\end{theorem}
	
	\begin{proof}
In the following, we first show that \Cref{ALGTemporalPaths1label} returns a temporal resolving set $R$ of $\mathcal{P=}(P_n,\lambda)$. After that, we prove that $R$ is minimum-sized and finally, that the algorithm has linear-time complexity.

First of all, consider vertices in $\rea(r_1)$. Note that if there are vertices $u,v_b\in \rea(r_1)$ such that $\dist_t(r_1,u)=\dist_t(r_1,v_b)$, then one of them is on the left side of $r_1$ and other one is on its right side. Let us assume, without loss of generality, that $u$ is on the left side of $r_1$.  Moreover, no third vertex $w$ can have $\dist_t(r_1,u)=\dist_t(r_1,v_b)=\dist_t(r_1,w)$. Let us assume that $v_b$ is the vertex with the smallest index on the right side of $r_1$ such that it is not separated by $r_1$ from some other vertex (in this case, from $u$).
		
In this case, \Cref{ALGTemporalPaths1label} has chosen $w=v_b$ on Step~\ref{Alg1Else2} and set $v=w$ after that on Step~\ref{Alg1Setvw}. Hence, in the following while-loop, we choose $r_2$ as the rightmost vertex which reaches $v(=v_b)$. Furthermore, $r_2$ cannot reach $u$. Indeed, since $r_1$ cannot separate 
$u$ and $v_b$, there are two edges with the same time label on the path from $u$ to $v_b$. Therefore, on the path from $r_2$ to $u$, there are two edges with the same time label. Hence, $r_2$ separates $v_b$ from $u$. Consequently, if there were any other vertices in $\rea^{\{r_1\}}(r_1)$ which were not separated by $r_1$, then they would be separated by $r_2$. A similar argument works for all pairs $r_i,r_{i+1}$. Note that in the end, either we choose $r_{|R|}=v_n$, or $r_{|R|}$ separates every vertex in $\rea^R(r_{|R|})$. Hence, we eventually enter the if-clause on Step~\ref{Alg1If2} and return $R$.
		
		\smallskip
		We now show that there does not exist any resolving set of smaller size than $R$ in $\mathcal{P}$. We do this by induction on the number $n$ of vertices. First of all, \Cref{ALGTemporalPaths1label} outputs a temporal resolving set of size $1$ when $n\in\{1,2\}$, which is optimal. Thus, we assume from now on that it outputs a minimum-size temporal resolving set for $n\leq n'$.
		
		Let $n=n'+1$, and $R$ be the temporal resolving set constructed using \Cref{ALGTemporalPaths1label} on $\mathcal{P}=(P_n,\lambda)$. Assume first that $r_1$ separates every vertex in $\rea(r_1)$. Observe that if $|R|=1$, then it is minimum-size. Hence, we may assume that $\mathcal{P}'=\mathcal{P} \setminus \rea(r_1)$ (where $\mathcal{G} \setminus V'$ for a temporal graph $\mathcal{G}=(V,E,\lambda)$ denotes the temporal subgraph $\mathcal{G'}=(V \setminus V',E \setminus \{uv~:~u \in V' \mbox{ or } v\in V'\},\lambda)$) is nonempty. By induction, \Cref{ALGTemporalPaths1label} outputs a minimum-size temporal resolving set $R\setminus\{r_1\}$ of $\mathcal{P}'$. Observe that $\ell(r_1)\cap \rea(w)=\emptyset$ for any $w$ on the rightside of $r_1$. Moreover,  we require at least one vertex in set $\ell(r_1)\cup\{r_1\}$ in any resolving set of $\mathcal{P}$ to reach vertex $v_1$. Observe that \Cref{ALGTemporalPaths1label} returns the temporal resolving set $R\setminus\{r_1\}$ for $\mathcal{P}\setminus \rea(r_1)$. By induction, set $R\setminus\{r_1\}$ has minimum size. By \Cref{lemSubpath}, we require in any temporal path containing $\mathcal{P}\setminus\rea(r_1)$ (and having $v_n$ as a leaf) at least $|R|-1$ vertices. Furthermore, by our observations, we require in a set $\ell(r_1)\cup\{r_1\}$ at least one vertex to reach $v_1$. Furthermore, since these vertices do not reach any vertex in $V(\mathcal{P})\setminus\rea(r_1)$, we require at least $|R|$ vertices in a resolving set of $\mathcal{P}$, as claimed.

		Assume next that there are vertices $v_\ell$ and $v_r$ in $\rea(r_1)$ which are not separated by $r_1$. Further assume that $v_\ell$ (resp. $v_r$) is on the left (resp. right) side of $r_1$. Consequently, $r_2$ is the rightmost vertex which reaches $v_r$. We show that if $v_\ell\neq v_1$, then the claim follows. Assume that $\ell\geq2$. Consider the temporal path $\mathcal{P}'=\mathcal{P} \setminus \{v_1\}$. Note that $r_1$ is the rightmost vertex which reaches $v_2$. Moreover, $r_1$ does not separate vertices $v_\ell$ and $v_r$. Thus, \Cref{ALGTemporalPaths1label} outputs $R$ as a resolving set for $\mathcal{P}'$. By our induction hypothesis, $R$ is minimum-size. By \Cref{lemSubpath}, we know that $R$ is at least as large as a minimum-size temporal resolving set of $\mathcal{P}'$. Thus, $R$ is also a minimum-size temporal resolving set for $\mathcal{P}$. Note that this implies that $r_1$ separates every pair of vertices in $\rea(r_1)$ except for $(v_1,v_r)$.   
		
		We now need to analyze several cases depending on whether $t_{r-1}\leq t_{r}$ and whether $t_r\leq t_{r+1}$. We distinguish four cases and for all of them, we conclude 
		that at least $|R|$ vertices are necessary in a temporal resolving set of $\mathcal{P}$. 
		
		\noindent Our aim is to conclude that temporal resolving set is of size at least $|R|$ in $\mathcal{P}$ in all of the cases.
		
		Assume first that $t_{r-1}\leq t_{r}$ and $t_r\leq t_{r+1}$. Thus, $r_2=v_{r+1}$.  Furthermore, $r_2$ separates every vertex in $\rea(r_2)\setminus\{v_r\}$. Indeed since $r_1$ reaches $v_r$, we have $\rea(r_2)\setminus\{v_r\}=\rea(r_2)\setminus \ell(r_2)$, and $r_2$ pairwise separates every vertex on its right side.
		Denote by $\mathcal{P}_r$ the temporal subpath $\mathcal{P} \setminus (\rea(r_1) \cup \rea(r_2))$. Note that \Cref{ALGTemporalPaths1label} outputs $R\setminus\{r_1,r_2\}$ as a temporal resolving set of $\mathcal{P}_r$ and, by induction, this has minimum cardinality in $\mathcal{P}_r$. Note that if $|R|=2$, then $R$ has the smallest possible size in $\mathcal{P}$. Hence, we assume that $|R|>2$. By \Cref{lemSubpath}, any subpath of $\mathcal{P}$ containing $\mathcal{P}'$ and leaf $v_n$ has temporal resolving number at least $|R|-2$.
		Furthermore, we require at least two vertices in a temporal resolving set of $\mathcal{P}$ among vertices $\ell(r_2)\cup\{r_2\}$. Note that none of these vertices reach any vertex in $V(\mathcal{P}')$. Thus, at least $|R|$ vertices are necessary in a temporal resolving set of $\mathcal{P}$ as claimed.
		
		Consider now the case with $t_{r-1}= t_{r}$ and $t_r> t_{r+1}$. We have $\rea(r_1)=\{v_1,\dots,v_r\}$. Consider subpath $\mathcal{P}'=\mathcal{P}\setminus\rea(r_1)$. Note that \Cref{ALGTemporalPaths1label} outputs $R\setminus\{r_1\}$ for $\mathcal{P}'$ since $r_2\neq v_{r+1}$ as $t_r>t_{r+1}$. Moreover, $R\setminus\{r_1\}$ is a minimum-sized temporal resolving set by induction assumption. By \Cref{lemSubpath}, any subpath of $\mathcal{P}$ containing $\mathcal{P}'$ and $v_n$ requires at least $|R|-1$ vertices in any temporal resolving set. Moreover, we require at least one vertex in $\ell(r_1)\cup\{r_1\}$ for any temporal resolving set of $\mathcal{P}$. Note that vertices in $\ell(r_1)\cup\{r_1\}$ do not reach vertices in $V(\mathcal{P}')$. Hence, we require at least $|R|$ vertices in any minimum-sized temporal resolving set of $\mathcal{P}$.
		
		Consider next the case with $t_{r-1}< t_{r}$ and $t_r> t_{r+1}$. We have $\rea(r_1)=\{v_1,\dots,v_{r+1}\}$. Consider $\mathcal{P}'$ such that $V(\mathcal{P}')=\{v_{r-1},\dots,v_n\}$, $E(\mathcal{P}')\subseteq E(\mathcal{P})$, $\lambda(v_{r-1}v_r)=t_r$ and $\lambda(v_{r+i}v_{r+i+1})=t_{r+i}$ for each $i\geq0$. Note that \Cref{ALGTemporalPaths1label} now outputs the temporal resolving set $R'=(R\cup\{v_r\})\setminus\{r_1\}$. By induction assumption, $R'$ has minimum cardinality in $\mathcal{P}'$. In particular, any temporal resolving set of $\mathcal{P}'$ requires one of the vertices $v_{r-1}$ or $v_r$ as these are the only vertices which reach $v_{r-1}$. Consider a temporal resolving set $R''$ of $\mathcal{P}$. Let $R^*=R''\cap \ell(v_{r+1})$. Observe that $(R''\setminus R^*)\cup\{v_r\}$ is a temporal resolving set of $\mathcal{P}'$. Hence, $|R''|-|R^*|+1\geq |R|$. Since we require at least one vertex in $R^*$ to reach $v_1$, we have $|R''|\geq|R|$. Therefore, $|R|$ has the minimum cardinality over temporal resolving sets of $\mathcal{P}$, as claimed.

		Next, we consider the case where $t_{r-1}> t_{r}$ and $t_r\neq t_{r+1}$. Observe that $v_{r+1}\not\in \rea(r_1)$. Since $v_\ell=v_1$, $r_1$  separates $v_{r+1}$ from other vertices in $\rea(r_1)$.
		Assume next that all time labels have even values. This has no effect on the temporal resolving set or the algorithm (we can multiply every time label by two without changing any reachability in the temporal path). We do the following modification to the time labeling of $\mathcal{P}$, obtaining path $\mathcal{P}_m$. We change  $t_r$ into $t'_r= t_r-1$. Note that since we assumed that every time label has even value, time label $t'_r$ has an odd value unlike all other time labels, and $t_{r-1}>t'_r$. Moreover if $t_{r+1}> t_r$, then $t_{r+1}> t_r'$ and if $t_{r+1}< t_r$, then $t_{r+1}< t_r'$. Note that in this change we maintain $v_r$ and $v_1$ unseparable by $r_1$ and every set $\rea(v_i)$ remains unchanged.
		
		Observe that \Cref{ALGTemporalPaths1label} returns the same temporal resolving set $R$ for $\mathcal{P}_m$ since $r_2$ is still the rightmost vertex which reaches $v_r$. Moreover, any temporal resolving set for $\mathcal{P}$ remains as a temporal resolving set for $\mathcal{P}_m$ since sets $\rea(v_i)$ remain unchanged and, since the label $t'_{r}$ is the only odd label, it cannot cause any two vertices to become unseparated. Let $\mathcal{P}_m'=\mathcal{P}_m \setminus \rea(r_1)$. Observe that \Cref{ALGTemporalPaths1label} outputs a temporal resolving set $R'=R\setminus \{r_1\}$ for $\mathcal{P}_m'$. Furthermore, by the induction hypothesis, it is minimum-size. Consider next a temporal resolving set $R''$ of $\mathcal{P}_m$. Assume first that there are at least two vertices in $\rea(r_1)\cap R''$ and let us denote the rightmost of them by $w$. Denote $\mathcal{P}_w=\mathcal{P}_m \setminus\ell(w)$. Note that $R''\setminus \ell(w)$ is a temporal resolving set of $\mathcal{P}_w$. Observe that $\mathcal{P}_m'$ is a subgraph of $\mathcal{P}_w$. Thus, by \Cref{lemSubpath}, we have $|R''\setminus\ell(w)|\geq |R|-1$ and $|R''|\geq|R|$. Assume then that we have $|R''\cap \rea(r_1)|=1$. Let $w_1\in R''\cap \rea(r_1)$. Note that if $v_r\in \rea(w_1)$, then $w_1$ does not separate $v_\ell$ and $v_r$. Thus, some other vertex in $R''\setminus\rea(r_1)$ reaches $v_r$ and due to the odd time label,  $R''\setminus \rea(r_1)$ is a temporal resolving set of $\mathcal{P}_m'$. Hence, by \Cref{lemSubpath}, we have $|R''|\geq|R|$, allowing us to conclude that $R$ is a minimum-size temporal resolving set of $\mathcal{P}_m$. If we have a temporal resolving set $R^*$ of $\mathcal{P}$ with $|R^*|<|R|$, then $R^*$ is also a temporal resolving set of $\mathcal{P}_m$ with $|R^*|<|R|$, a contradiction.
		
		As the last case, we consider $t_{r-1}> t_{r}= t_{r+1}$. Again, observe that $v_{r+1}\not\in \rea(r_1).$ Furthermore, $v_{r+1}=r_2$. Hence, together, $r_1$ and $r_2$ separate all vertices in $\rea(r_1)\cup\rea(r_2)$. Furthermore, since at least one vertex in $\ell(r_1)\cup\{r_1\}$ is required, note that if $|R|=2$, then it has minimum size in $\mathcal{P}$. Thus, assume that $|R|\geq3$. Let $\mathcal{P}'=\mathcal{P}\setminus(\rea(r_1)\cup\rea(r_2))$. Note that set $R\setminus\{r_1,r_2\}$ is a minimum-size temporal resolving set of $\mathcal{P}'$ since $r_1$ and $r_2$ separate all vertices in $\rea(r_1)\cup\rea(r_2)$. Furthermore, \Cref{ALGTemporalPaths1label} outputs set $R\setminus\{r_1,r_2\}$ for $\mathcal{P}'$. By the induction assumption, set $R\setminus\{r_1,r_2\}$ is a smallest temporal resolving set of $\mathcal{P}'$. Hence, by \Cref{lemSubpath}, any subpath of $\mathcal{P}$ containing $\mathcal{P}'$ and leaf $v_n$ requires at least $|R|-2$ vertices in any temporal resolving set. Furthermore, at least two vertices $\ell(r_2)\cup\{r_2\}$ are required in any temporal resolving set of $\mathcal{P}$, and these vertices do not reach any vertex in $V(\mathcal{P}')$. Thus, $\mathcal{P}$ does not have any temporal resolving set with cardinality less than $|R|$, as claimed.
			
		Finally, we show that \Cref{ALGTemporalPaths1label} has linear time complexity. First of all, the while-loop ends at some point since every temporal path has a temporal resolving set by taking every vertex in the underlying path. Secondly, Step~\ref{Alg1Setsvi} uses at most $i+1-a$ comparisons and in total at most $2n$ comparisons. In Step~\ref{Alg1If1}, observe that sets $\rea^R(s)$ do not overlap. Thus, each vertex is considered only once. Moreover, the time labels on the left and the right side of $s$ are ordered from small to large. Thus, checking if each time label has a unique value can be done in linear time on $|\rea^R(s)|$. Again, in Step~\ref{Alg1Then1}, the sets $\rea(s)\setminus \ell (s)$ do not overlap. Thus, this step takes at most linear time on $n$ in total. Finally, all the other steps take at most constant time. Hence, the algorithm has linear-time complexity.
	\end{proof}
	
	The following lemma gives some structure on the minimum-size temporal resolving sets of temporal paths, with respect to the one output by \Cref{ALGTemporalPaths1label}. In particular, it states that the constructed temporal resolving set $R$ places each vertex in $R$ as far away from the leaf $v_1$ as possible. It will be used in the case of subdivided stars, allowing us to reuse \Cref{ALGTemporalPaths1label} to find a partial solution.
	
	\begin{lemma}
		\label{LemPathGreedyness}
		Let $\mathcal{P}$ be a temporal path on $n\geq2$ vertices with $1$-labeling $\lambda$ on vertices $v_1,\dots,v_n$ and edges $v_iv_{i+1}$ for $ 1\leq i\leq n-1$ where $v_i$ is on the left side of $v_{i+1}$ for each $i$. Let $R=\{r_1,\dots,r_{|R|} \}$ be the temporal resolving set output by \Cref{ALGTemporalPaths1label}, where $r_i$ is on the left side of $r_{i+1}$ for each $i$. Let $R'=\{r_1',\dots,r_{|R|}'\}$ be another temporal resolving set of $\mathcal{P}$. We have $r_i'\in\ell(r_i)\cup\{r_i\}$ for each $i$.
	\end{lemma}
	
	\begin{proof}
		Recall that by \Cref{The:TempPath1label}, Algorithm \ref{ALGTemporalPaths1label} returns a temporal resolving set of minimum size for $\mathcal{P}$.
		Consider first the case with $|R|=1$. Notice that $r_1$ is the rightmost vertex which reaches $v_1$. Thus, the claim holds in this case. Consequently, by the same argument, we have that $r_1'\in \ell(r_1)\cup\{r_1\}$ even when $|R|>1$. Assume next that the claim does not hold for some $\mathcal{P}$, $R$ and $R'$. Furthermore, let $r_i'\not\in \ell(r_i)\cup\{r_i\}$ and $r_j'\in \ell(r_j)\cup\{r_j\}$ for every $j<i$. Let us assume first that $r_{i-1}=r_{i-1}'$. Denote by $R_j=\{r_1,\dots,r_j\}$ and $R_j'=\{r_1',\dots,r_j'\}$ for any $j\leq|R|$. 
		
		Consider first the case where $r_{i-1}$ separates all vertices in $\rea^{R_{i-1}}(r_{i-1})$. Let $w$ be the leftmost vertex which is not reached by $R_{i-1}$. In this case, \Cref{ALGTemporalPaths1label} chooses $r_i$ as the rightmost vertex which reaches $w$. Since also $r_i'$ reaches $w$, we have $r_i'\in\ell(r_i)\cup\{r_i\}$, a contradiction. Hence, there exist vertices $u,w\in \rea^{R_{i-1}}(r_{i-1})$ which are not separated $r_{i-1}$. Let $u\in \ell(r_{i-1})$ and $w\in \rea(r_{i-1})\setminus\ell(r_{i-1})$. Furthermore, we assume that $w$ is the leftmost vertex with these properties. Notice that $r_i$ is the rightmost vertex which reaches $w$. Moreover, if $u,w\in \rea^{R_{i-1}'}(r_{i-1}')$ and $r_{i-1}'\neq u$, then $r_{i-1}'$ does not separate $u$ and $w$. Thus, vertex $r_{i}'$ reaches $w$. However, since $r_i$ was the rightmost such vertex, we have $r_i'\in \ell(r_i)\cup\{r_i\}$, a contradiction. Furthermore, if $r_{i-1}'=u$, then $w\not\in \rea(r_{i-1}')$. Indeed, there are two edges with the same time label on the path from $u$ to $w$ since $r_{i-1}$ does not separate them. Therefore, to reach $w$, we have $w\in \rea(r_i')$. Again, this leads to a contradiction since $r_i$ was the rightmost vertex reaching $w$. Hence, $u\not\in \rea^{R_{i-1}'}(r_{i-1}')$ and $r_{i-1}'$ is on the rightside of $u$. In particular, this implies that $u\in\rea(r_{i-2}')$. Furthermore, we have $u\not\in \rea(r_{i-2})$. Indeed, otherwise $r_{i-2}$ or $ r_{i-1}$ would separate $w$ and $u$. Moreover, $u$ is on the rightside of $r_{i-2}$. Since $u\in\rea(r_{i-2}')$ but $u\not\in \rea(r_{i-2})$, we have $r_{i-2}'$ on the rightside of $r_{i-2}$. However, this is a contradiction with the minimality of $i$. Therefore, the claim holds. 
	\end{proof}
		
	\subsection{Temporal stars}\label{sec-polyStars}
	
	In this subsection, we give polynomial-time algorithms for finding minimum-size temporal resolving sets for temporal stars with $1$-labeling $\lambda$ and temporal subdivided stars with $1$-labeling using only values $1$ and $2$. For both these classes, the central vertex is denoted by $c$.
	
	\begin{theorem}
		\label{thm-polyStar}
		Let $S$ be a star, $\lambda$ be a 1-labeling, and $X$ be a maximum-size set of leaves of $S$ such that, for distinct $u,v \in X$, $\lambda(cu) \neq \lambda(cv)$. The set $V(S)\setminus X$ is a minimum-size  temporal resolving set of $(S,\lambda)$.
	\end{theorem}
	
	\begin{proof}
    Throughout this proof, we denote by $S_n$ the star with central vertex $c$ and $n$ leaves $v_1,\dots,v_n$. We furthermore assume that the time labeling $\lambda$ uses time labels from 1 to $m$ and does not contain any gaps (we can remove gaps without changing reachability); since $\lambda(cv_i)$ contains only one integer, we will use it to denote the integer it contains by abuse of notation. We can further assume without loss of generality that $\lambda(cv_i)\leq \lambda(cv_{i+1})$ for each $1\leq i\leq n-1$. Let $VL_j = \{v_i~:~\lambda(cv_i)=j\}$ and $L_j = | VL_j |$. We denote by $VL_j'$ a set $VL_j\setminus \{v\}$ where $v$ is an arbitrary vertex of $VL_j$.
		
		We first prove that $R=V(S_n) \setminus X$ from the statement of the theorem is a temporal resolving set of $\mathcal{S}=(S,\lambda)$. First, we have $V(S_n)\subseteq \rea(c)$ and $c\in R$.
		Furthermore, by the definition of the sets $VL_j'$, $c$ separates every vertex in $V(S_n)\setminus R$. Hence, $R$ is a temporal resolving set of $\mathcal{S}$.
		
		Next, let us prove the minimality of $R$. First, assume that for some $j$, we have $u,v\in VL_j\setminus R$. However, we have $\lambda(cu)=\lambda(cv)=j$. Thus, these vertices cannot be separated and we have $|VL_j\setminus R|\leq 1$ for each $j$. Assume then that $|VL_j\setminus R|= 1$ for each $1\leq j\leq m$ and that $c\not\in R$. However, now there is a vertex $v\in VL_1\setminus R$ and the only vertices that can reach $v$ are $v$ and $c$ since $1$ is the smallest label. Thus, no vertex in $R$ reaches $v$ and hence, we cannot simultaneously have $|VL_j\setminus R|= 1$ for each $1\leq j\leq m$ and  $c\not\in R$. Therefore, $R$ has minimum cardinality.
	\end{proof}
	
		We next consider subdivided stars together with a $1$-labeling $\lambda$ using only values $1$ and $2$. In particular, we present a polynomial-time algorithm for this case. Let $SS_\Delta$ be a subdivided star of maximum degree $\Delta$ and central vertex $c$. By a \emph{branch} of $SS_\Delta$, we mean a maximal path starting from the central vertex $c$ but not containing $c$ itself. Branches are denoted by $B_1, \dots, B_\Delta$ and the leaves by $\ell_1,\dots, \ell_\Delta$. The three vertices in branch $B_i$ closest to $c$ are denoted by $v_i$, $u_i$ and $w_i$, respectively (note that $u_i$ and $w_i$ may not exist if $|B_i| \leq 1$ or $|B_i| \leq 2$; in theses cases, one of these three vertices may also be $\ell_i$). We further assume that branches are ordered so that $\lambda(cv_i)\leq \lambda(cv_{i+1})$ for each $i\leq \Delta-1$. We assume that $\lambda(cv_i)=1$ if and only if $i\leq I_1$ where $0\leq I_1\leq \Delta$ and denote $B^1=\bigcup_{i=1}^{I_1}B_i$ and $B^2=\bigcup_{i=I_1+1}^{\Delta}B_i$.
	
	The basic principles of Algorithm \ref{ALGTemporalSubStars1label}  on temporal subdivided star $\mathcal{SS}_\Delta=(SS_\Delta,\lambda)$ are as follows:
	\begin{enumerate}
		\item For each path from $\ell_i$ to $c$, we apply Algorithm \ref{ALGTemporalPaths1label} starting from $\ell_i$ creating temporal resolving set $R_i$ for this path. Denote $(\bigcup_{i=1}^\Delta R_i)\setminus \{c\}$ by $R'$.
		\item Let $ B_c=\bigcup_{c\in R_i} B_i$. If $B_c=\emptyset$, then we are done.
		\item We construct an $O(n^5)$ number of sets $R''$ containing vertices in the vicinity of $c$ such that $R'\cup R''$ is a temporal resolving set, and select the smallest such set.
	\end{enumerate}

	The following theorem shows that Algorithm \ref{ALGTemporalSubStars1label} returns a minimum-size temporal resolving set in polynomial time for a given temporal subdivided star using only values $1$ and $2$ in its $1$-labeling. In order to do this, we show that there is a minimum-size temporal resolving set of the form $R' \cup R''$ as constructed above.

	\begin{algorithm}[th!]
		\caption{Computing in polynomial-time a minimum-size temporal resolving set of a temporal subdivided star $\mathcal{SS}_\Delta=(SS_\Delta,\lambda)$ such that, for any edge $e$, either $\lambda(e)=1$ or $\lambda(e)=2$.}
		\label{ALGTemporalSubStars1label}
		
		\SetKwInOut{KwIn}{Input}
		\SetKwInOut{KwOut}{Output}
		\KwIn{Subdivided star $SS_\Delta$ of degree $\Delta\geq3$ together with  time labeling $\lambda$ for each edge $e\in E(SS_\Delta)$ such that $\lambda(e)\in\{1,2\}$.}
		\KwOut{A minimum-size temporal resolving set $R$ of $\mathcal{SS}_\Delta$.}
		
		For each path from $\ell_i$ to $c$, we create temporal resolving set $R_i$ using \Cref{ALGTemporalPaths1label}.\label{alg2R_i}
		
		Let $R'=(\bigcup_{i=1}^\Delta R_i)\setminus \{c\}$.\label{alg2R'}
		
		Let $ B_c=\bigcup_{c\in R_i} B_i$.\label{alg2B_c}
		
		\If{$B_c=\emptyset$\label{alg2If1}}
		{\Return $R'$.\label{alg2Return1}}
		
		\For{$j\in \{1,2\}$ and $1\leq i\leq \Delta$\label{alg2For1}}
		{
			\If{
				$v_i\in V(B_c)$ and $B_i\in B^j$, for some $i\leq \Delta$, and $R_i\setminus \{c\}$ does not separate $v_i$ from some other vertex in $B_i$ but $v_i\in \rea(R_i\setminus\{c\})$\label{alg2If2}}
			{ add $v_i$ to $Q_j$. \label{alg2Then1}}
		}
		
		\For{$1 \leq i\leq \Delta$\label{alg2For2}} {
			Remove  branch $B_i$ from $B_c$ if $Q_2\cap V(B_i)\neq \emptyset$.\label{alg2Then2}
		}
		
		Let $r=|\{i\mid Q_1\cap V(B_i)\neq \emptyset\}|$.\label{alg2r}
		
		\For{$0\leq j\leq 2$\label{alg2For3}}{
			\For{each vertex set $R''$ such that $R''\subseteq \rea(c)\cap (V(B_c)\cup\{c\})\setminus R'$, $|R''\cap V(B_i)|\leq 1$ for every $B_i\in B_c$, $|R''\cap Q_1|\leq 1$, $|R''\cap\{a\mid \dist(c,a)=2\}|\leq 1$ and $|R''|= |B_c|-1-r+j$\label{alg2For4}}{
				\If{$R'\cup R''$ is a temporal resolving set of $\mathcal{SS}_\Delta$\label{alg2If3}}{
					\Return $R'\cup R''$ \label{alg2Return2}
				}
			}
		}
		
	\end{algorithm}

	\begin{theorem}
		\label{thm:polyalgSubStars1label12}
		Given a subdivided star $SS_\Delta$ of maximum degree $\Delta\geq3$ and a 1-labeling of edges $\lambda$ restricted to values $1$ and $2$,
		\Cref{ALGTemporalSubStars1label} returns a minimum-size temporal resolving set of $\mathcal{SS}_\Delta=(SS_\Delta,\lambda)$ in polynomial-time on the number of vertices.
	\end{theorem}
	
	\begin{proof}
		We consider that the central vertex $c$ is the rightmost vertex and the leaves are the leftmost vertices of their branches whenever reasoning on \Cref{ALGTemporalPaths1label}.
		We first show that Algorithm \ref{ALGTemporalSubStars1label} returns a temporal resolving set, then that the temporal resolving set has minimum possible size, and finally that the algorithm operates in polynomial time.
		
		First of all, by \Cref{The:TempPath1label}, Algorithm \ref{ALGTemporalPaths1label} returns a minimum-size temporal resolving set $R_i$ for a path from $\ell_i$ to $c$ in Step~\ref{alg2R_i}. Denote this path together with its time labeling $\lambda$ by $\mathcal{P}_i$. Let us first consider the set $R'$ which is returned in Step~\ref{alg2Return1}. If we enter the if-clause on Step~\ref{alg2If1}, then we have $B_c=\emptyset$ and $c\not\in \bigcup_{i=1}^\Delta R_i$. Thus, $R'=\bigcup_{i=1}^\Delta R_i$. Since each $R_i$ is a temporal resolving set for path from $\ell_i$ to $c$, each vertex in the substar is reached by some vertex in $R'$. Moreover, each of these sets separates all vertices within the same path $\mathcal{P}_i$. Thus, if two vertices $u\in V(B_i)$ and $v \in V(B_j)$ are not separated by $R'$, then we have $i\neq j$. However, since $c\not\in R'$, vertices $u$ and $v$ are reached by at least two different vertices in $R'$, from two different sets $R_i$ and $R_j$. Let $u\in \rea(r_u)$ for some $r_u\in R_i$. If $v\not \in \rea(r_u)$, then $r_u$ separates $u$ and $v$. If there is a path from $r_u$ to $u$ to $v$, then $r_u$ separates vertices $u$ and $v$. Thus, $r_u$ is on the path from $u$ to $v$. Similarly, we have $r_v\in R_j$ between vertices $v$ and $u$.   
		Thus, neither of $u$ nor $v$ can be vertex $c$. Consequently, also $c\not\in R'$ is on the path from $u$ to $v$. 
		Moreover, since our labeling consists of values $1$ and $2$, vertex $r_v$ cannot reach $u$. We conclude that $u$ and $v$ are separated. Hence, if $R'$ is returned in Step~\ref{alg2Return1}, then it is a temporal resolving set of the subdivided star.
		
		In the for-clause between Steps~\ref{alg2For1} and~\ref{alg2Then1}, we consider each vertex $v\in N(c)$ in some branch $B_i\in B^j$ which is reached by $R_i\setminus\{c\}$ but not separated from some other vertex in $B_i$. Each such vertex is added to set $Q_j$ for $j$ corresponding to $B^j$. Observe that any vertex in $N[c]$ which reaches $v\in Q_j$, also separates it from all other vertices together with $R_i$.
		
		In the for-clause between Steps~\ref{alg2For2} and~\ref{alg2Then2}, every branch which has non-empty intersection with $Q_2$ is removed from $B_c$. Notice that if $v\in Q_2$, then $v\in N(c)$. Hence, if we remove a branch $B_i$ during this step from $B_c$, then we had $\rea(c)\cap V(B_i)=\{v\}$.
		
		In the for-clause between Steps~\ref{alg2For3} and~\ref{alg2Return2}, we construct a temporal resolving set for the substar. Observe that in particular the for-clause can always find, with $j=3$, the set $R''$ which contains $c$ and the first vertex in every branch in $B_c$ which has empty intersection with $ Q_1$. Together with $R'$, this forms a temporal resolving set of the substar (although not always minimum-size). Indeed, recall that $R_i\cup \{c\}$ is a temporal resolving set for $\mathcal{P}_i$. Thus, $(R_i\cup\{v_i\})\setminus\{c\}$ is also a temporal resolving set for $\mathcal{P}_i \setminus \{c\}$. Furthermore, each $v_i$ can only reach the first vertices of other branches in $B^2$ and no vertices in branches in $B^1$. These vertices are either in $R'\cup R''$, in $Q_2$, or in a branch which has temporal resolving set $R_j\setminus\{c\}$. Since we have $c\in R''$, vertices in $Q_j$ for $j\in\{1,2\}$ are reached and separated from other vertices. By these considerations, the only vertex pairs which might not be separated belong to two different branches and do not belong to $Q_j$. Let $v\in V(B_i)$ and $u\in V(B_j)$ be two vertices which are not separated. Note that if $u\not\in N(c)$, then it cannot be reached by any vertex in $V(B_i)$ and vice versa. Thus, $v,u\in N(c)\setminus R''$. However, now $v$ ($u$) is reached by some vertex $r_v$ ($r_u$) in $R'\cap V(B_i)$ ($R'\cap V(B_j)$). Moreover, we have $u\not\in \rea(r_v)$ ($v\not\in \rea(r_u)$). Thus, $u$ and $v$ are separated and $R'\cup R''$ is a temporal resolving set of the substar.\medskip
		
		Let us next show that the returned temporal resolving set has the minimum size. Consider first set $R'$ on Step~\ref{alg2Return1} and assume that it is returned. In this case, we have $B_c=\emptyset$ and thus, $c\not\in \bigcup_{i=1}^\Delta R_i$. Consider path $\mathcal{P}_i-\rea(c)$. Note that for this path, \Cref{ALGTemporalPaths1label} returns a temporal resolving set of cardinality equal to $|R_i|$. Furthermore, by \Cref{The:TempPath1label}, this set has minimum possible cardinality. By \Cref{LemPathGreedyness}, there is no temporal resolving set of cardinality $|R_i|$ for path $\mathcal{P}_i$ that contains vertex $c$. Moreover, we require for each path $\mathcal{P}_i-\rea(c)$ at least $|R_i|$ vertices in a temporal resolving set. Thus, $R'$ has the smallest possible cardinality.
		
		Let us then show that if $R'$ is not returned on Step~\ref{alg2Return1}, then $R'\cup R''$ has the minimum cardinality for a temporal resolving set in $\mathcal{SS}_\Delta$. By our earlier considerations, $R'\cup R''$ is a temporal resolving set for some suitable $R''$. Let us first show that we may assume that $R'$ is a subset of some minimum-size temporal resolving set of $\mathcal{SS}_\Delta$. First of all, in Step~\ref{alg2R_i}, \Cref{ALGTemporalPaths1label} returns a minimum-size temporal resolving set $R_i$ for each path $\mathcal{P}_i$.  Furthermore, for path $\mathcal{P}_i-\rea(c)\setminus (R_i\setminus\{c\})$, \Cref{ALGTemporalPaths1label} returns the temporal resolving set $R_i\setminus\{c\}$. Indeed, $\mathcal{P}_i-\rea(c)\setminus (R_i\setminus\{c\})$ is a path since \Cref{ALGTemporalPaths1label} always picks the last vertex to reach previously considered vertices and we only use labels $1$ and $2$.
		When traversing the path from a leaf to the center, the algorithm does not consider how to separate/reach vertices ahead of it. Let $r_i\in R_i\setminus \{c\}$ be the vertex closest to $c$ in $R_i\setminus \{c\}$. If $c\not\in R_i$, then by \Cref{LemPathGreedyness},
		any temporal resolving set containing a vertex in $\mathcal{P}_i$ closer to $c$ than vertex $r_i$ has at least $|R_i|+1$ vertices. Moreover, if $c\in R_i$, then any temporal resolving set containing a vertex in $\mathcal{P}_i$ closer to $c$ than vertex $r_i$ has at least $|R_i|$ vertices other than $c$. By these considerations, $r_i$ is the vertex closest to center $c$ other than $c$ which can be contained in any temporal resolving set of cardinality $|R_i\setminus\{c\}|$ over all subpaths of $\mathcal{P}_i$ containing leaf $\ell_i$. 
		
		Furthermore, if $R$ is a resolving set of $\mathcal{SS}_\Delta$ and $R_{P_i}=R\cap V(\mathcal{P}_i^r)$ (where $\mathcal{P}_i^r$ is the subpath of $\mathcal{P}_i$ from $\ell_i$ to $r_i$), then if $|R_{P_i}|=|R_i\setminus\{c\}|$, then $R_P=(R\setminus R_{P_i})\cup (R_i\setminus \{c\})$ is also a temporal resolving set of $\mathcal{SS}_\Delta$ and if $|R_{P_i}|>R_i\setminus\{c\}$,  then $R_P'=\{c\}\cup(R\setminus R_{P_i})\cup (R_i\setminus \{c\})$ is also a temporal resolving set of $\mathcal{SS}_\Delta$.  Consider first the case with $|R_{P_i}|=|R_i\setminus\{c\}|$.
		In this case, $R_i$ reaches and separates every vertex in  $\mathcal{P}_i^r$ and $r_i$ is a leaf of $\mathcal{P}_i^r$. Assume on the contrary that $R_P$ is not a temporal resolving set of $\mathcal{SS}_\Delta$. In this case, there are (at least) two vertices in $\rea(r_i)$ which are not separated by $r_i$. Note that one of them is on the right side (call this $u$) of $r_i$ and one is on the left side (call this $v$). Moreover, if $r_i\not\in R$, then we have some $r_i'\in R$ on the left side of $r_i$ which is the rightmost vertex of $R_{P_i}$. If $r_i'$ separates $u$ and $v$, then  $u\not\in \rea(r_i')$.
		In this case, a vertex in $R\setminus R_{P_i}$ reaches $u$ and separates it from $v$.
		Thus, $r_i\in R$. Let $r_{i-1}'$ ($r_{i-1}$) be the first vertex on the left side of $r_i$ in $R$ ($R_i$). By \Cref{LemPathGreedyness}, the vertex $r_{i-1}'$ is not on the right  side of $r_{i-1}$. Thus, $R_{P_i}$ does not separate vertices $u$ and $v$ but set $R\setminus R_{P_i}$ does separate them. Therefore, $R_P$ also separates $u$ and $v$. Thus, $R_P$ is a temporal resolving set of $\mathcal{P}$.   Moreover, we have $|R_P|=|R|$.

		Let us next consider the case with $|R_{P_i}|>|R_i\setminus\{c\}|$. Recall that $R_i$ is a temporal resolving set of $\mathcal{P}_i$. Thus, if $R_P'=R_P\cup\{c\}$ does not separate some vertices, then those vertices belong to different branches of $\mathcal{SS}_\Delta$ and are only reached by $c$. However, in this case they are also not reached by $R_{P_i}$ and thus, not separated by $R$, a contradiction. Thus, $R_P'$ is a temporal resolving set of $\mathcal{SS}_\Delta$ with $|R_P'|\leq |R|$.
		Therefore, we have shown that $R'$ or $R'\cup\{c\}$ is a subset of some minimum-size temporal resolving set of $\mathcal{SS}_\Delta$.

		In the following, we consider $R''$ from Step~\ref{alg2For4}.
		Observe that there are at most $|B_c|-r$ vertices $v_i$ in $N(c)$ which are not reached by any vertex in $R'$. Let us have $|B_c\cap B^1|=c_1$ and $|B_c\cap B^2|=c_2$.    
		Out of these vertices, note that $c$ cannot separate any two vertices $v_i,v_j\in N(c)\cap V(B_c)\cap V(B^h)$ for $h\in\{1,2\}$. Moreover, to reach every vertex $v_i$ in $B_c\cap B^1$, we require either $c$ to be in $R''$ or some vertex from $B_i$ to be in set $R''$. Thus, we have $|R''|\geq |B_c|-r-1$. We have earlier shown that for $|R''|=|B_c|-r+1$, there always exists a temporal resolving set. Thus, $|B_c|-r-1\leq |R''|\leq |B_c|-r+1$. Let us next show that we may assume that $R''\subseteq \rea(c)\cap (V(B_c)\cup\{c\})\setminus R'$. We immediately note, that assuming $R'\cap R''=\emptyset$ does not affect on $R'\cup R''$. Assume first on the contrary, that for every minimum  set $R^*$ such that $R'\cup R^*$ is a temporal resolving set of $\mathcal{SS}_\Delta$, we have $w\in R^*$ and $w\not\in \rea(c)$. Furthermore, assume that among such sets, $R^*$ contains the smallest possible number of vertices outside of $\rea(c)$. 
		Since $R'$ separates vertices of $\mathcal{SS}_\Delta-\rea(c)$, we have a vertex $v\in \rea(w)\cap\rea(c)$. Otherwise, $R'\cup R^*\setminus\{w\}$ would be a temporal resolving set of smaller size. If $\rea(w)\cap \rea(c)\setminus\{c\}=v$, then $R'\cup (R^*\setminus\{w\})\cup\{v\}$ is a temporal resolving set of $\mathcal{SS}_\Delta$, a contradiction. Thus, $u,v\in \rea(w)\cap \rea(c)$. Note that $u$ and $v$ belong to the same branch. Assume that $\dist_t(c,u)=2$ and $\dist_t(c,v)=1$. In this case, $R'\cup (R^*\setminus\{w\})\cup\{u\}$ is a temporal resolving set of $\mathcal{SS}_\Delta$, a contradiction.  Thus, we may assume that $R''\subseteq \rea(c)$. 
		
		Let us next show that we may assume that $R''\subseteq V(B_c)\cup \{c\}$. Suppose next on the contrary that for every minimum-size  set $R^*$ such that $R'\cup R^*$ is a temporal resolving set of $\mathcal{SS}_\Delta$, we have $w\in R^*$ and $w\not\in V(B_c)\cap \{c\}$. Furthermore, assume that among such sets, $R^*$ contains the smallest possible number of vertices outside of $V(B_c)\cap \{c\}$.By earlier arguments, we may assume that $R^*\subseteq \rea(c)$. Hence, $\dist_t(c,w)\in\{1,2\}$. 
        Observe that $R'$ reaches and separates every vertex in the branch $w$ is located since  $w\not\in V(B_c)$. 
  
        Assume next that $\dist_t(c,w)=2$, then the only vertex, which is possibly separated from some other vertex by $w$ but not by $R'\cup (R^*\setminus\{w\})$, is $c$. Thus, $R'\cup (R^*\setminus\{w\})\cup\{c\}$ is a temporal resolving set of $\mathcal{SS}_\Delta$, a contradiction.

        Assume then that $\dist_t(c,w)=1$. Thus, $w\in N(c)$. Now, $\rea(c)\cap\rea(w)=\{c\}\cup\{w\}\cup(N(c)\cap \{u\mid \dist_t(c,u)=2\})\cup(N(w)\cap\{u\mid \dist_t(c,u)=2\})$. Note that vertices in $\{w\}\cup (N(w)\cap\{u\mid \dist_t(c,u)=2\})$ are already separated from other vertices by $R'$. Consider a pair of vertices $u_c,u_w$ which is separated by $w$ but not by $c$ and let $u_c\in\rea(c)\setminus \rea(w)$ and  $v_w\in\rea(w)\cap \rea(c)$. Furthermore, we have $\dist_t(c,v_w)=\dist_t(c,u_c)=2$. Hence, $v_w\in N(c)\cap \{u\mid \dist_t(c,u)=2\}$ and $u_c\in \{u\mid \dist_t(c,u)=2 = \dist(c,u)\})$. 
        Moreover, since the pair $ v_w,u_c$ is not separated by $R'\cup R^*\setminus\{w\}$, there is another a vertex $x\neq w$ in $R'\cup R^*$ which reaches $u_c$ but does not separate $ v_w$ and $u_c$ since $R'\cup R^*$ is a temporal resolving set of $\mathcal{SS}_\Delta$. Let $u_c$ belong to branch $B_k$. Observe that there are exactly two options for $x$ since $\dist_t(x,v_w)=\dist_t(x,u_c)\leq2$: $c$ and the single vertex, say $v_c$, on the path from $c$ to $u_c$. Note that $v_c\not\in R'$ by Lemma \ref{LemPathGreedyness} since Algorithm \ref{ALGTemporalPaths1label} would place $c$ into the temporal resolving set rather than vertex $v_c$. Therefore,  $B_k\in B_c$. Let assume next that $R'\cup R^*\setminus\{w\}$ separates all other vertex pairs except for $u_c,v_w$. However, in this case $\{u_c\}\cup R'\cup R^*\setminus\{w\}$ is a temporal resolving set of $\mathcal{SS}_\Delta$, a contradiction. Hence, we assume that there is also another vertex pair $u_c',v_w'$ which is not separated by $R'\cup R^*\setminus\{w\}$. Note that the same restrictions apply for $u_c'$ and $v_w'$ as we have for $u_c$ and $v_w$. Furthermore, if we cannot choose $u_c'\neq u_c$, then $\{u_c\}\cup R'\cup R^*\setminus\{w\}$ is a temporal resolving set of $\mathcal{SS}_\Delta$. Thus, let us assume that $u_c'\neq u_c$. Since $u_c,u_c'\not\in \rea(w)$ and since $R'\cup R^*$ is a temporal resolving, there is a vertex $z\in R'\cup R^*\setminus\{w\}$ which separates $u_c$ and $u_c'$. Assume without loss of generality that $u_c'\in \rea(z)$. Again, $\{u_c\}\cup R'\cup R^*\setminus\{w\}$ is a temporal resolving set of $\mathcal{SS}_\Delta$.  Therefore,  $R''\subseteq V(B_c)\cup \{c\}$ and hence, $R''\subseteq \rea(c)\cap(V(B_c)\cap \{c\})$.
        
		Let us next show that we may assume that $|R''\cap V(B_i)|\leq1$ for each $i$. Observe first that since $R''\subseteq \rea(c)$, we have $|R''\cap V(B_i)|\leq2$. 
		Suppose to the contrary, that we have $|R''\cap V(B_i)|=2$ for some $i$ and $R''$ has the smallest number of such branches among all minimum-size temporal resolving sets $R'\cup R''$. Thus, we have $v_i,u_i\in R''$, $\lambda(cv_i)=1$ and $\lambda(v_iu_i)=2$. Since $R''\subseteq V(B_c)\cup \{c\}$, we have $B_i\in B_c$. Since $R'\cup R''\setminus\{v_i\}$ is not a temporal resolving set, vertex $v_i$ separates a pair of vertices $v_j,u_k$ in $V(B_c)$ with $\dist_t(c,v_j)=\dist_t(c,u_k)=2$. Note that $R'\cup \{u_i\}$ separates every vertex in $B_i$ since $u_i\not\in R'$.
        Since $v_i$ separates $u_k$ and $v_j$, exactly $v_j\in \rea(v_i)$ and $\dist(c,v_j)=1$, $\dist_t(c,v_j)=2$ and $\dist(c,u_k)=\dist_t(c,u_k)=2$. 
        Since $R'\cup R''$ is a temporal resolving set, there is a vertex in $R'\cup R''$ which reaches $u_k$ and since it does not separate $u_k$ and $v_j$, that vertex is either $c$ or $v_k$. 
    Furthermore, $R'\cup R''\setminus\{v_i\}$ separates $u_k$ from each vertex except $v_j$. Indeed, if we had $v_p$ which is not separated from $u_k$, then $v_p$ and $v_j$ would not be separated by $R'\cup R''$ since $v_i$ does not separate $v_p$ and $v_j$. Similarly, if we had $u_p$ which is not separated from $u_k$, then they are also not separated by $R'\cup R''$. Thus,  $v_j$ and $u_k$  are the only vertices which are not separated by $R'\cup R''$. Furthermore, we may observe that the only vertices which might be in $R'\cup R''$ at temporal distance one from $c$ are $v_i$ and $v_k$ since any other vertex would separate $u_k$ and $v_j$. Consider now set $\{v_j\}\cup R'\cup R''\setminus\{v_i\}$. Since $v_j$ separates $v_j$ and $u_k$ and vertices $u_k$ and either $c$ or $v_k$ together reach every vertex which is reached by $v_i$, the set $\{v_j\}\cup R'\cup R''\setminus\{v_i\}$ is a temporal resolving set, a contradiction.

		Let us next show that we may assume that $|R''\cap Q_1|\leq1$. Recall that $Q_1\subseteq N(c)\cap \{v\mid \dist_t(c,v)=1\}$. Assume on the contrary that we have $v_i,v_j\in R''\cap Q_1$. Then, $R=R'\cup (R''\setminus \{v_i\})\cup\{c\}$ is a temporal resolving set of $\mathcal{SS}_\Delta.$  Indeed, recall that $c$, together with $R'$, separates vertices in $Q_1$ from all other vertices. Furthermore, the only vertex which $v_i$ could separate which is not separated by $c$ or $v_j$ is $u_j$ if $\dist_t(c,u_j)=2$. However, since, by the definition of $Q_1$, the set $R'$ reaches vertex $v_j\in Q_1$, the set $R'$ also reaches $u_j$. Hence, $R'$ separates $u_j$ from other vertices at temporal distance~2 from $c$, and $R$ is a temporal resolving set of $\mathcal{SS}_\Delta$.

		Let us next show that we may assume that $|R''\cap\{a \mid \dist(c,a)=2\} |\leq1$. We assume on the contrary that we have two vertices $u_i,u_j\in R''\cap\{a \mid \dist(c,a)=2\}$. By our assumptions, we have $R''\subseteq \rea(c)$. Hence, $\dist_t(c,u_i)=\dist_t(c,u_j)=2$. We claim that  $R=R'\cup (R''\setminus \{u_i,u_j\})\cup\{v_i,v_j\}$ is a temporal resolving set of $\mathcal{SS}_\Delta$. Indeed, the only vertices in $\rea(c)$ which are reached by $u_i$ and $u_j$ are in the set $\{v_i,v_j,u_i,u_j\}$. Out of these, $v_i$ and $v_j$ separate themselves. Furthermore, $u_i$ is separated from vertices in $\{u\mid \dist(c,u)=\dist_t(c,u)=2\}$ by $v_i$ and from vertices in  $\{v\mid \dist(c,v)=1$ and $\dist_t(c,v)=2\}$ by $v_j$. Similar arguments hold for $u_j$.  Hence, $R$ is a temporal resolving set of $\mathcal{SS}_\Delta$. 
		
		Since we obtain a temporal resolving set in Steps~\ref{alg2For3} to~\ref{alg2Return2} and there is a minimum-size temporal resolving set of the form $R'\cup R''$, we conclude that we find a minimum-size temporal resolving set in Steps~\ref{alg2For3} to~\ref{alg2Return2} (or in Step~\ref{alg2Return1}).
		
		Let us finally show that \Cref{ALGTemporalSubStars1label} ends in polynomial time. Recall that it takes polynomial time on the number of vertices $n$, especially when the time labels are in the set $\{1,2\}$, to check if a given set is a temporal resolving set of a graph. Furthermore, in Step~\ref{alg2R_i}, \Cref{ALGTemporalPaths1label} works in linear-time and we need to apply it at most $n$ times. Steps~\ref{alg2R'} to~\ref{alg2r} clearly operate in polynomial-time. In Steps~\ref{alg2For3} to~\ref{alg2Return2}, we test for a given $j\in \{0,1,2\}$, at most $ \binom{|B_c|+1}{3-j}\cdot r\cdot|B_c|\in O(n^5)$ times if a given set is a temporal resolving set. Hence, \Cref{ALGTemporalSubStars1label} works in polynomial time, completing the proof.
	\end{proof}
	
	\section{Combinatorial results for $p$-periodic 1-labelings}\label{SecPartLabel} 
 
	In this section, we focus on $p$-periodic 1-labelings, \emph{i.e.}, the time labelings where each edge appears once in every interval of $p$ consecutive time-steps.
    Given a temporal graph $\mathcal{G}=(G,\lambda)$ where $\lambda$ is a $p$-periodic 1-labeling, we denote by $M_p(\mathcal{G})$ the temporal resolving number (\emph{i.e.}, the minimum size of a temporal resolving set) of $\mathcal{G}$. Furthermore, if $\lambda(e)=\{i,i+p,i+2p,\ldots\}$, then, by abuse of notation, we denote $\lambda(e)=i$.
	When $p=1$ or $\lambda(e)$ is the same for every edge $e$, those are exactly the usual resolving sets. 
    We bound the minimum size of temporal resolving sets for several graph classes.

	Note that, in this section, reachability is trivially assured (since the time-steps repeat indefinitely and the considered graphs are connected), so to prove that a given set is a temporal resolving set, we only need to prove that it is separating.
	
	\begin{theorem}
		\label{PropkperiodicPath}
		Let $P_n$ be a path on $n$ vertices, $\lambda$ be a $p$-periodic 1-labeling, and $\mathcal{P}=(P_n,\lambda)$. We have $M_p(\mathcal{P})=1$.
	\end{theorem}

	\begin{proof}
		Let $u_1,\ldots,u_n$ be the vertices of $P_n$, with edges $u_iu_{i+1}$ for $1 \leq i \leq n-1$. Let $R=\{u_1\}$. For $i \in \{2,\ldots,n-1\}$, $u_i$ is reached from $u_1$ strictly before $u_{i+1}$. The same reasoning works for $R=\{u_n\}$. Hence, any of the leaves (clearly)  forms a minimum-size temporal resolving set.
	\end{proof}

    \begin{remark}
        \Cref{PropkperiodicPath} also holds for general $p$-periodic labelings.
    \end{remark}
	
	In particular, the proof of Theorem \ref{PropkperiodicPath} implies that in a temporal tree $\mathcal{T}$ with $p$-periodic $1$-labeling if we have a path from $r$ to $u$ to $v$, then $r$ separates vertices $u$ and $v$. In this case, we say that vertices $u$ and $v$ are \textit{path-separated} (by $r$). In the following two theorems, we introduce combinatorial results for some simple graph classes. Note that the following theorem implies a polynomial-time algorithm by testing each vertex set containing at most two vertices. 
	
	\begin{theorem}
		\label{PropkperiodicCycle} 
		Let $C_n$ be a cycle on $n$ vertices, $\lambda$ be a $p$-periodic 1-labeling, and $\mathcal{C}=(C_n,\lambda)$. We have $1\leq M_p(\mathcal{C})\leq 2$.
	\end{theorem}
	
	\begin{proof}
		Let $C_n$ be a cycle on $n$ vertices, $\lambda$ be a $p$-periodic 1-labeling, and $\mathcal{C}=(C_n,\lambda)$. Let $e=uv$ be a locally maximally labeled edge of $\mathcal{C}$, \emph{i.e}, an edge with $\lambda(e)$ such that the adjacent edges have labels at most $\lambda(e)$ (such an edge has to exist). We claim that $u$ and $v$ form a resolving set.
		
		Suppose for a contradiction that there are two vertices, $x$ and $y$, not separated by $u$ and $v$. That means that they have the same temporal distance from $u$ and from $v$. 
		
		Let us first consider temporal distances from $u$. By $e$ being locally maximally labeled edge, it must be that precisely one of the paths from $u$ to $x$ and from $u$ to $y$ attaining the minimal temporal distance must go through $e$. Otherwise, $u$ path-separates $x$ and $y$.
		Without loss of generality, let $e$ be on the path from $u$ to $y$. Thus, $\dist_t(v,y)=\dist_t(u,y)-p$. 
		However, by the same reasoning as above, path attaining the temporal distance from $v$ to $x$ must now use edge $e$ and thus, $\dist_t(u,x)=\dist_t(v,x)-p$. Since $\dist_t(u,x)=\dist_t(v,x)$. We have $\dist_t(v,y)=\dist_t(v,x)-2p$, a contradiction.
	\end{proof}

The previous theorem shows how to solve periodic cycles with $1$-labelings. However, the case with periodic $k$-labelings is open for cycles. 
    \begin{problem}
		\label{ProbPeriodicCycle} 
        Does there exist a $p$-periodic time labeling $\lambda$ and some $n$ such that for $\mathcal{C}=(C_n,\lambda)$ we have $M_p(\mathcal{C})\geq3$?
	\end{problem}

 In the following theorem, we give tight bounds for temporal resolving number of temporal complete graphs with $p$-periodic $1$-labeling.
	
	\begin{theorem}
		\label{PropkperiodicComplete}
		Let $K_n$ be a complete graph on $n=b+p^b$ vertices with $b \geq 1$, $\lambda$ be a $p$-periodic 1-labeling, and $\mathcal{K}=(K_n,\lambda)$. We have $b\leq M_p(\mathcal{K})\leq n-1$ and both bounds are tight.
	\end{theorem}
	
	\begin{proof}
		The upper bound is trivial. To prove its tightness, consider $\mathcal{K}=(K_n,\lambda)$ with $\lambda$ assigning the same time label to all edges. All pairs of vertices in the graph are twins and therefore, we have to take at least one of the vertices in all such pairs. This results in taking $n-1$ vertices.
		
		Let us prove the lower bound. For a contradiction, suppose there would be less than $b$ vertices in a temporal resolving set, say $b'<b$. There are still more than $p^b$ vertices to separate but there are just $p^{b'}$ possible distance vectors to the vertices outside of our chosen set, less than $p^b$. This means that some vertices have to share a distance vector and thus, they are not separated, a contradiction. 
		
		We shall now construct a complete graph $K_n$ with a $p$-periodic 1-labeling which attains the lower bound. To this end, take a subset $B$ of $b$ vertices in a fixed order. Then, give every vertex $v\in V(K_n)\setminus B$ a unique $p$-ary tuple $\ell(v)$ of length $b$ containing values from $1$ to $p$. We label an edge between $i$-th element of $B$ and $v\in V(K_n)\setminus B$ by $j$ if the $i$-th position of $\ell(v)$ is $j$. 
		The remaining edges of the graph, \emph{i.e.}, edges running between the vertices of $B$ and between the vertices outside of $B$ will get label $p$. Clearly, the vertices of $B$ now form a resolving set since the constructed $b$-tuples are precisely the vectors of distances.
	\end{proof}
	
	The following lemma will help us to simplify the remaining results on trees, as we will be able to consider only temporal resolving sets composed of leaves.
	
	\begin{lemma}
		\label{PropkperiodicLeaves}
		Let $T$ be a tree, $\lambda$ be a $p$-periodic 1-labeling, $\mathcal{T}=(T,\lambda)$ with $M_p(\mathcal{T}) \geq 2$. There is a temporal resolving set of $\mathcal{T}$ of size $M_p(\mathcal{T})$ containing only leaves of~$T$.
	\end{lemma}
	
	\begin{proof}
		Let $T$ be a tree, $\lambda$ a $p$-periodic 1-labeling, and $\mathcal{T}=(T,\lambda)$.
		Suppose we have a minimum temporal resolving set $R$ of size at least two with the minimum number of non-leaves.
		
		If every leaf is in the resolving set, then every two non-leaf vertices are path-separated by a leaf and thus there can be no non-leaves in $R$, as otherwise we can construct a smaller temporal resolving set. Hence, we may assume that not all leaves do belong to $R$.
		
		Assume next that there is at least one non-leaf vertex $v \in R$. We root the tree in $v$. Our aim is to find a suitable leaf $\ell \not\in R$, which will be exchanged with $v$ in order to get a same-size temporal resolving set with a greater number of leaves.
		Let us denote, for a branch $B$, the label of the edge between $v$ and the vertex adjacent to $v$ in $V(B)$ by $\lambda(B)$. Let $m=\min\{\lambda(B)\mid B\text{ is a branch}\}$. Denote by $B^m$ the set of branches $B$ which have $\lambda(B)=m$. If there is a  branch $B\in B^m$ such that $R\cap V(B)=\emptyset$, then we choose $\ell$ as a leaf of $B$. Otherwise, if there exists a branch $B'$ such that $R\cap V(B')=\emptyset$, then we select $\ell$ as a leaf of $B'$ and if such $B'$ does not exist, then we select $\ell$ as an arbitrary leaf not in $R$. Let $R'=(R\cup\{\ell\})\setminus\{v\}$. Notice that since $R$ is a temporal resolving set of $\mathcal{T}$, after this process we have $R'\cap V(B)\neq \emptyset$ for each $B\in B^m$. In the following, for $x,y \in V(T)$, we always assume that $x\in V(B_x)\cap R'$ where $B_x$ is the branch in which $x$ resides, $B_x\in B^m$, and $y\in R\setminus V(B_x)$.
		
		We now consider different cases for a pair of vertices $a,b\in V(T)$ separated by $v$ (which is in $R$). In particular, we show that $R'$ separates them as well.\medskip
		
		\noindent \emph{\textbf{Case 1}. $a\in V(B),b \in V(B')$ for any $B,B'\not\in B^m$}: Note that we may have $B=B'$. Recall that $x\in V(B_x)\cap R'$ and $B_x\in B^m$. We have $\dist_t(x,a)=\dist_t(x,v)+\dist_t(v,a)-m$ and $\dist_t(x,b)=\dist_t(x,v)+\dist_t(v,b)-m$. Since $v$ separates $a$ and $b$, we have $\dist_t(v,a)\neq \dist_t(v,b)$. Hence, $\dist_t(x,a)\neq \dist_t(x,b)$. \medskip
		
		\noindent \emph{\textbf{Case 2}. $a,b \in V(B)$ for $B\in B^m$}: Let $y\in V(B_y)\cap R'$ and $B_y\neq B$. We have $\dist_t(y,a)=\dist_t(y,v)+\dist_t(v,a)+(p-\lambda(B_y))$ and $\dist_t(y,b)=\dist_t(y,v)+\dist_t(v,b)+(p-\lambda(B_y))$ since 
		$\lambda(B)\leq \lambda(B_y)$. Since $v$ separates $a$ and $b$, we have $\dist_t(v,a)\neq \dist_t(v,b)$. Hence, $\dist_t(y,a)\neq \dist_t(y,b)$.\medskip

		\noindent \emph{\textbf{Case 3}. $a \in V(B), b \not\in V(B)$ for $B\in B^m$}: 
		Since $B\in B^m$, we have $V(B)\cap R'\neq \emptyset$.
		Assume without loss of generality, that $B=B_x$. Recall that $x\in V(B_x)\cap R'$.
		Let $c$ be the last vertex on the common paths from $x$ to $a$ and $x$ to $b$ (possibly $c=x$). We have $c\in V(B_x)$. Note that if $c=a$, then $x$ path-separates $a$ and $b$. Let us denote by $\lambda(a)$ ($\lambda(b)$) the time label of the first edge on the path from $c$ to $a$ (to $b$). Furthermore, denote by $\lambda(X)$ the label of the last edge on the path from $x$ to $c$. We let $\lambda(X)=0$ if $ x=c$.
		
		When we have $\lambda(X)<\min\{\lambda(a),\lambda(b)\}$ or $\lambda(X)\geq\max\{\lambda(a),\lambda(b)\}$, vertex $x$ separates $a$ and $b$ if and only if $c$ separates $a$ and $b$. If $\lambda(a)\leq \lambda(X)<\lambda(b)$, then vertex $x$ separates vertices $a$ and $b$ if and only if $\dist_t(c,b)\neq \dist_t(c,a)+p$.
		
		Similarly, if $\lambda(b)\leq \lambda(X)<\lambda(a)$, then vertex $x$ separates vertices $a$ and $b$ if and only if $\dist_t(c,a)\neq \dist_t(c,b)+p$. Moreover, in all three subcases
		we have $\dist_t(c,a)\leq \dist_t(v,a)$ and $\dist_t(c,b)\geq \dist_t(v,b)$. Observe that if $ x$ does not separate vertices $a$ and $b$, then the time label of the last edge on the paths from $v$ to $a$ and from $v$ to $b$ is the same.
		
		We conclude that $\dist_t(v,a)=\dist_t(v,b)+h\cdot p$ for some integer $h$. We denote by $B_b$ the branch in which $b$ is located. We shall now prove a crucial claim, saying that in all the cases, $h$ is a positive integer.
		
		\begin{claim}
			In all four aforementioned cases, if $x$ does not separate $a$ and $b$, then, we have $h\geq1$.
		\end{claim}
		\begin{proof} We shall divide the proof according to the cases which might occur.\medskip
			
			\noindent \emph{\textbf{Subcase 3.a}. $\lambda(X)<\min\{\lambda(a),\lambda(b)\}$ or $\lambda(X)\geq\max\{\lambda(a),\lambda(b)\}$}: 
			In this case, $x$ separates $a$ and $b$ if and only if $c$ separates $a$ and $b$. We have $\dist_t(c,a)\leq \dist_t(v,a)$ and $\dist_t(c,b)\geq \dist_t(v,b)$  and at least one of these two inequalities is strict (otherwise $c$ separates $a$ and $b$). Hence, if $c$ does not separate $a$ and $b$, then $\dist_t(v,a)=\dist_t(v,b)+h\cdot p$ for some positive integer $h$. Indeed, the time label of the last edge on the paths from $v$ to $a$ and from $v$ to $b$ is identical since $c$ does not separate these vertices.
			
			\medskip \noindent \emph{\textbf{Subcase 3.b}. $\lambda(b)\leq \lambda(X)<\lambda(a)$}: 
			In this case, $x$ separates vertices $a$ and $b$ if and only if $\dist_t(c,a)\neq \dist_t(c,b)+p$. Hence, we assume that $\dist_t(c,a)= \dist_t(c,b)+p$. We have 
			$\dist_t(v,a)=\dist_t(v,b)+h\cdot p$ 
			for some integer $h\neq 0$. Recall that we have  $\dist_t(c,b)\geq \dist_t(v,b)$ and $\dist_t(c,b)+p=\dist_t(c,a)\leq \dist_t(v,a)$. Hence, $\dist_t(v,b)+p\leq \dist_t(c,a)\leq \dist_t(v,a)$. Thus, $h\geq2$.
			
			\medskip \noindent \emph{\textbf{Subcase 3.c}. $\lambda(a)\leq \lambda(X)<\lambda(b)$}: 
			In this case, $x$ separates vertices $a$ and $b$ if and only if $\dist_t(c,b)\neq \dist_t(c,a)+p$. Assume that this is not the case and  $\dist_t(c,b)= \dist_t(c,a)+p$.  Recall that we have $\dist_t(v,a)=\dist_t(v,b)+h\cdot p$ for some integer $h\neq 0$. We have $\dist_t(v,a)=\dist_t(v,c)+\dist_t(c,a)+(p-\lambda(b))$ because $\lambda(b)>\lambda(a)$. Since $\dist_t(v,c)\geq\lambda(b)$, we have $\dist_t(v,a)\geq \dist_t(c,a)+p$.      
			Recall that, we have  $\dist_t(c,b)\geq \dist_t(v,b)$. Further, 
			$$\dist_t(v,b)+h\cdot p=\dist_t(v,a)\geq \dist_t(c,a)+p=\dist_t(c,b)\geq\dist_t(v,b).$$  Finally, since $\dist_t(v,a)\neq \dist_t(v,b)$, we have $h\geq1$.
			
			With this, we have proved the claim for all the subcases.
		\end{proof}
		
		We next proceed with the proof of Case 3 together with the assumption $\dist_t(v,a)=\dist_t(v,b)+h\cdot p$ for a positive integer $h$.
		
		Consider now a vertex $r\in R'\cap V(B_r)$ for $B_r\not\in B_x\cup B_b$ (if such a $r$ exists). We have \begin{align*}
			\dist_t(r,a)&=\dist_t(r,v)+\dist_t(v,a)+(p-\lambda(B_r))\\
			&=\dist_t(r,v)+\dist_t(v,b)+(h+1)\cdot p-\lambda(B_r)\\
			&\geq \dist_t(r ,v)+\dist_t(v,b)+ 2p-\lambda(B_r ).
		\end{align*}
		Moreover, if $\lambda(B_b)\leq\lambda(B_r )$, then $\dist_t(r ,b)=\dist_t(r ,v)+\dist_t(v,b)+(p-\lambda(B_r ))<\dist_t(r ,a)$. Furthermore, if $\lambda(B_b)>\lambda(B_r )$, then $\dist_t(r ,b)=\dist_t(r ,v)+\dist_t(v,b)-\lambda(B_r ) < \dist_t(r ,a)$. Hence, $r$ separates $a$ and $b$ if $r$ exists.
		
		If we have a vertex $y\in V(B_b)\cap R'$, then we denote by $c'$ the last common vertex on the path from $y$ to $a$ and from $y$ to $b$. Since $B_x\in B^m$, we have \begin{equation}\label{eqDistc'a}
			\dist_t(c',a)=\dist_t(v,a)+\dist_t(c',v)+(p-\lambda(B_b))\geq \dist_t(v,a)+p.
		\end{equation}  We denote by $\lambda(a')$ and $\lambda(b')$ the time label of the first edge on the path from $c'$ to $a$ and from $c'$ to $b$, respectively. Furthermore, denote by $\lambda(Y)$ the label of the last edge on the path from $y$ to $c'$. We let $\lambda(Y)=0$ if $ y=c'$. Similarly to $x$, when we have $\lambda(Y)<\min\{\lambda(a'),\lambda(b')\}$ or $\lambda(Y)\geq\max\{\lambda(a'),\lambda(b')\}$, vertex $y$ separates $a$ and $b$ if and only if $c'$ separates $a$ and $b$. If $\lambda(a')\leq \lambda(Y)<\lambda(b')$, then vertex $y$ separates vertices $a$ and $b$ if and only if $\dist_t(c',b)\neq \dist_t(c',a)+p$. Similarly, if $\lambda(b')\leq \lambda(Y)<\lambda(a')$, then vertex $y$ separates vertices $a$ and $b$ if and only if $\dist_t(c',a)\neq \dist_t(c',b)+p$. In particular, if $\dist_t(c',a)>\dist_t(c',b)+p$, then $y$ separates $a$ and $b$.

		We have $$\dist_t(c',b)\leq\dist_t(v,b)\leq  \dist_t(v,a)-p\leq \dist_t(c',a)-2p.$$ The second inequality is due to Claim and the last inequality is due to Equation (\ref{eqDistc'a}). 
		Hence, $\dist_t(c',a)>\dist_t(c',b)+p$. Therefore, $y$ separates $a$ and $b$ for every ordering of $\lambda(a'),\lambda(b')$ and $\lambda(Y)$.

		\medskip
		\noindent \emph{\textbf{Case 4}. $a \in V(B), b =v$}: Let $y\in R'\setminus V(B)$. In this case, $y$ simply path-separates $a$ and $b$.\medskip
		
		Therefore, $R'$ separates all vertices which are separated by $R$ in all cases. Since $|R|=|R'|$, this contradicts the assumption that $R$ has the minimum possible number of non-leaves. Hence, the statement follows.
	\end{proof}

        We note that it is unclear if the previous lemma also holds for temporal trees with periodic $k$-labeling.
        
        \begin{problem}
		\label{ProbPeriodicLeaves} 
        Does Lemma \ref{PropkperiodicLeaves} hold for temporal trees with periodic $k$-labeling for $k\geq2$?
	\end{problem}
	
	\Cref{PropkperiodicLeaves} gives us an FPT algorithm for \temporalMD in trees with respect to the number of leaves. It also allows us to prove a brute force-like polynomial-time algorithm for temporal subdivided stars with a $p$-periodic 1-labeling for fixed $p$. This shows that \temporalMD is in XP with respect to the period of the time labeling in this context.
	\begin{theorem}\label{ThmSubStarPol}
		\temporalMD is polynomial-time solvable in temporal subdivided stars with a $p$-periodic labeling for a fixed constant $p$.  
	\end{theorem}

	\begin{proof}
		Let $S$ be a subdivided star with central vertex $c$ and $\ell$ leaves, $\lambda$ be a $p$-periodic 1-labeling, and $\mathcal{S}=(S,\lambda)$. By \Cref{PropkperiodicLeaves}, there exists a minimum-size temporal resolving set in $\mathcal{S}$ which contains only leaves of $S$, if $M_p(\mathcal S)>1$.
		
		Note that if some branches, say $b$ of them, share the same time label on their edges incident with $c$, then, to separate vertices in these branches, we need to select at least $b-1$ leaves in a temporal resolving set. Since there are at most $p$ distinct labels, we have, for any temporal resolving set $R$ of $\mathcal{S}$, $|R|\geq \ell-p$. Furthermore, checking whether a given vertex set is a temporal resolving set of $\mathcal{S}$ can be done in polynomial time.
		
		If $n$ is the order of $S$, then there are $\binom{\ell}{\ell-p}=\binom{\ell}{p}\in O(n^p)$ ways to select a vertex set containing exactly $\ell-p$ leaves. Since $p$ is a fixed constant, we can check all these sets in polynomial time. If any of them is a temporal resolving set, then we found a minimum-size temporal resolving set of $\mathcal{S}$. Otherwise, we iterate the process by examining all vertex sets of size $\ell-p+1$, $\ell-p+2$, $\ldots$, until we find a temporal resolving set. The process will eventually stop as the set containing all leaves of $S$ is a temporal resolving set of $\mathcal{S}$. Thus, we need to check at most $\sum_{i=0}^p\binom{\ell}{p-i}$ vertex sets, that is, we need to do $O(n^p)$ polynomial-time operations.
	\end{proof}
	
	We end this section with two combinatorial results for other two subclasses of temporal trees: subdivided stars and complete binary trees.
	
	\begin{theorem}
		\label{PropkperiodicStars}
		Let $S$ be a subdivided star on $\ell\geq2$ leaves, $\lambda$ be a $p$-periodic 1-labeling, and $\mathcal{S}=(S,\lambda)$. We have $\max(1,\ell-p)\leq M_p(\mathcal{S})\leq \ell-1$ and both bounds are tight.
	\end{theorem}
	
	\begin{proof}
		Suppose that we have a subdivided star $\mathcal{S}$ with a $p$-periodic 1-labeling. Denote by $c$ the central vertex of $S$. By \Cref{PropkperiodicLeaves}, if $M_p(\mathcal{S}) \geq 2$, then there exists a minimum-size temporal resolving set of $S$ containing only leaves. Again, as in the proof of \Cref{ThmSubStarPol}, one can argue by a simple application of the pigeonhole principle applied to the distinct labels of edges incident with $c$ that at least $\ell - p$ leaves have to be chosen in order to separate vertices in $N(c)$. For $\ell - p < 1$, a trivial bound saying that at least one vertex has to be chosen is clearly better.
		
		The lower bound is tight, as we can take a star on $\ell$ leaves such that first $p$ edges will have different labels and the remaining $\ell - p$ leaves will all share the same label, say $1$. As there are $\ell - p + 1$ vertices being mutually twins, any minimum-size temporal resolving set has to be of size at least $\ell-p$.
		
		For the upper bound, we show that taking any $\ell - 1$ leaves is enough to form a temporal resolving set $R$, taking \Cref{PropkperiodicLeaves} into account. Indeed, a leaf $\ell\in R$ path-separates every vertex in the same branch from each other and from vertices in other branches. Furthermore, $\ell$ also path-separates vertices in the single branch with leaf outside of $R$ from each other. Thus, $R$ separates all vertices in $\mathcal{S}.$
		
		The tightness can be exemplified by a subdivided star with $\ell$ leaves with exactly same labels on all its edges. In this case, all leaves are twins and we are forced to take all but one to separate them.
	\end{proof}
	
	\begin{theorem}
		\label{PropkperiodicBinTrees}
		Let $T$ be a complete binary tree on $2^n-1$ vertices, $\lambda$ be a $2$-periodic 1-labeling, and $\mathcal{T}=(T,\lambda)$. We have $2^{n-3}\leq M_2(\mathcal{T})\leq 2^{n-2}$. Both bounds are tight.
	\end{theorem}
 
	\begin{proof}
		The \emph{center} of a tree is a vertex $v$ such that the  maximum distance between the vertex $v$ and any other vertex of the tree is minimal. (In this case, the center is a unique vertex.)  For convenience, we consider all trees in this proof as rooted in the center.  Based on the distance from the center, we say that vertices at distance $i$ from the center are \emph{on level $i$}. For a given $i$, \emph{higher} levels are levels from $0$ to $i-1$, and \emph{lower} levels are those with value at least $i+1$. A complete binary tree on $2^n-1$ vertices has $2^{n-1}$ leaves and they are on level $n-1$; the center is on level $0$. 
  We shall say that leaves are \emph{close} if they are at distance $2$ in the underlying tree. Finally, subtrees of order seven rooted in a vertex on level $n-3$ and induced by all the vertices under such a vertex are called \emph{essential}.
		
		Again, by \Cref{PropkperiodicLeaves}, we consider only sets of leaves as candidates for optimal temporal resolving sets if these are of size at least two.
		
		For the lower bound, let $\mathcal{T}=(T,\lambda)$, where $T$ is a complete binary tree on $2^n-1$ vertices and $\lambda$ is a $2$-periodic 1-labeling. We focus on essential subtrees. If we take less than $2^{n-3}$ vertices into our candidate set $R$, then necessarily at least one of the essential subtrees will have none of its leaves in $R$. Let us denote its root by $r$. In such a subtree, the four paths from $r$ to the leaves have to be labeled by all four possible combinations of labels $1$ and $2$, otherwise we easily find two vertices not separated by $R$. However, even if this is the case, none of the vertices from $R$ are able to separate the leaf with labels $1$ and $1$ on the path from $r$ (denote the leaf by $l_1$) and the leaf with labels $2$ and $1$ on the path from $r$ (denote this leaf by $l_2$). Indeed, the temporal distance of any vertex $c$ of $R$ to $l_1$ and $l_2$ is $d+3$ where $d$ is the temporal distance of $c$ to $r$. Thus at least $2^{n-3}$ leaves are needed in any temporal resolving set of $\mathcal{T}$.
		
		To show that the lower bound is tight, consider $\mathcal{T}$ with 
		a $2$-periodic 1-labeling $\lambda$, 
		constructed in a top-down fashion, proceeding level by level  we label one of the edges going to the lower level with $1$ and the other one with $2$. We shall form a temporal resolving set $R$ of size $2^{n-3}$ by taking one leaf from each essential subtree: the one connected to the root of the respective essential subtree by path labeled only with $1$. We have to prove that all vertices are now separated. Let us have two different vertices, say $u$ and $v$, outside of $R$ and suppose they are not path-separated. We distinguish the following cases.
		\begin{itemize}
			\item \emph{Vertices $u$ and $v$ are both on levels $n-3$ or higher}: In this case $u$ and $v$ are path separated and thus, this case cannot occur.
			\item \emph{Vertices $u$ and $v$ are on levels $n-2$ and/or $n-1$}: Either $u$ and $v$ are in the same essential subtree $\mathcal{T}_e$ and then, by a simple calculation, the temporal distances from the unique vertex from $R$ in $\mathcal{T}_e$ separate these two vertices. (Temporal distances in $\mathcal{T}_e$ range from $0$ to $6$ and they appear uniquely.) Assume then  that $u$ and $v$ are in different essential subtrees. Let $u\in V(\mathcal{T}_e)$ and $v\in V(\mathcal{T}_e')$. We observe that the distance from the leaf, say $l_e$ in $R\cap V(\mathcal{T}_e)$, to $v$ is at least 7. Thus, $l_e$ separates vertices $u$ and $v$.
			\item \emph{Vertex $u$ is on level at most $n-3$, while $v$ is on a level between $n-3$ and $n-1$, or vice versa}: In this case, $u$ and $v$ are path-separated unless we are in the special case of $\mathcal{T}$ having precisely $7$ vertices and being itself an essential subtree. However, we already know that separation is guaranteed in this case.
		\end{itemize}
  
		We have covered all the possible cases and hence the tightness follows.
		
		For the upper bound, we consider temporal tree $\mathcal{T}=(T,\lambda)$, where $\lambda$ is a $2$-periodic 1-labeling. We show that taking one of the leaves in each of the $2^{n-2}$ subtrees rooted in a vertex on level $n-2$ suffices to separate all vertices in $\mathcal{T}$. Let us denote such a set by $R$. Again, we have to be careful only about the pairs of vertices that are not path-separated. Suppose we have such a pair of different vertices $u,v \not\in R$ so that there is no vertex $c \in R$ such that either $u$ is on the path from $c$ to $v$, or $v$ is on the path from $u$ to $c$. Based on our choice of $R$, both $u$ and $v$ have to be leaves. We observe that distance from a vertex $r\in R$ \textit{close} to $u$ has distance in set $\{2,3,4\}$ while the distance to other leaves is at least $4$. However, if $\dist_t(r,v)=4$, then $\dist_t(r,u)\leq 3$. Thus, $R$ separates all vertices and the claim is proved.
		
		To show that the upper bound is tight, consider
		$\mathcal{T}=(T,\lambda)$ with a $2$-periodic 1-labeling $\lambda$ where all edges have the same label. 
		Take any subset $R'$ of leaves of size at most $2^{n-2} - 1$. In such a way, there has to be at least one subtree rooted at level $n-2$ out of all $2^{n-2}$ possible ones with no leaves in $R'$. The leaves of such a subtree are not separated by $R'$ and thus, the tightness is proved. 
	\end{proof}

	\section{Conclusion}\label{SecConclusion}

    We have introduced the notion of temporal resolving sets. Our hardness results show, that \temporalMD is NP-hard on temporal complete graphs with 2-labelings, on temporal subdivided stars with $2$-labelings, on temporal trees with only one vertex of degree 5 or greater with 2-labelings and edges appearing in consecutive time-steps. On the positive side, we have given a linear time algorithm for temporal paths with 1-labelings, a polynomial algorithm for temporal stars and temporal subdivided stars with $1$-labelings where each edge has label $1$ or $2$.  We have also presented multiple combinatorial and algorithmic results for temporal graphs with a $p$-periodic $1$-labeling, including the exact value of temporal resolving number paths, tight bounds for temporal cycles, temporal complete graphs, temporal subdivided stars and temporal complete binary trees. Those results also gave us FPT and XP parameters for trees wrt the number of leaves and for subdivided stars wrt the temporal period, respectively. The results are summarized in \Cref{tab-complexityResults} along with complexity results on standard and $k$-truncated resolving sets, which highlight the difficulty of \temporalMD compared to static versions of the problem.

    \begin{table}
        \centering
        \begin{tblr}{
				colspec={Q[c,5em]Q[c,7em]Q[c,7em]Q[c,20em]},
				rows={m},
				hlines,
				hline{1}={2pt},
				hline{2}={2pt}
			}
			& {Standard \\resolving set} & $k$-truncated resolving set & {\temporalMD} \\
			
			Trees & poly~\cite{chartrand2000resolvability,harary1976metric,khuller1996landmarks,slater1975leaves} & {NP-hard,\\ XP wrt $k$~\cite{gutkovich2023computing}} & {NP-hard (2-labelings with consecutive time labels): \Cref{thm-NPhardnessOnTrees}, \\ FPT wrt number of leaves ($p$-periodic 1-labeling): \Cref{PropkperiodicLeaves}} \\
			
			Subdivided stars & poly & open & {NP-hard (2-labeling): \Cref{thm-NPhardnessOnSubStars}, \\ poly (1-labeling and $\tmax=2$): \Cref{thm:polyalgSubStars1label12}, \\ XP wrt $p$ ($p$-periodic 1-labeling): \Cref{ThmSubStarPol}} \\
			
			Stars, Paths & trivial & poly~\cite{frongillo2022truncated} & poly (1-labeling): \Cref{The:TempPath1label,thm-polyStar} \\
			
			Complete graphs & trivial & trivial & {NP-hard (1-labeling and $\tmax=2$): \Cref{thm-completeGraphsLabels12}}
		\end{tblr}
        \caption{Summary of complexity results for standard, $k$-truncated and temporal resolving sets on the graph classes studied in this paper.}
        \label{tab-complexityResults}
    \end{table}

    Since periodic 1-labelings seem easier to solve than non-periodic ones, the following open problem seems the next step in the study of temporal resolving sets:
    \begin{problem}
		\label{ProbPeriodicTrees} 
        Is \temporalMD polynomial in temporal trees with periodic 1-labelings?
	\end{problem}
    Note that if there is a positive answer to Open Problem \ref{ProbPeriodicLeaves}, then generalizing Open Problem \ref{ProbPeriodicTrees} for periodic $k$-labelings becomes tempting.

    As a resolving set is equivalent with a temporal resolving set in a temporal graph $\mathcal{G}$ where every edge has time labels from $1$ to $\diam(G)$ and there are more positive algorithmic results for resolving sets (they can be solved in polynomial-time in some classes for which \temporalMD is NP-hard, such as complete graphs or trees) than temporal resolving sets, it might be possible that each edge having many time labels makes the problem easier. This problem might also be easier in general graphs with $p$-periodic labelings. Moreover, since minimum $k$-truncated resolving sets can be found in polynomial-time in trees when $k$ is fixed, a potential direction would be to study the parameterized complexity of temporal resolving sets, with the total number of available time-steps as a natural parameter (this parameter is unbounded in our NP-hardness reductions for trees). Other potential future research in this direction would be the study of \temporalMD when parameterized by temporal measures that, when bounded, restrict the changes to the graph at each time-step, or access to portions of the graph at some point in time, such as interval membership width~\cite{bumpus2023edge}.

    Another possible direction for future research is considering other kinds of temporal distances. For example, instead of considering the first time-step at which we may arrive, we could take into account the time-step at which we started the journey, that is, we could consider the total number of time-steps it takes to make the journey. Another option would be to consider the temporal distance from vertices to the temporal resolving set. This could yield new results as the temporal distance is not symmetric and would correspond to a setup where we are tracking an object. 
 	
	\bibliographystyle{plain}
	\bibliography{Temporal}
	
\end{document}